\documentclass{amsart}
\usepackage{amsmath,amsthm}
\usepackage{amsfonts,amssymb}
\usepackage{enumerate}
\usepackage{accents,color}
\usepackage{graphicx}

\newtheorem{thm}{Theorem}[section]
\newtheorem{cor}[thm]{Corollary}
\newtheorem{lem}[thm]{Lemma}
\newtheorem{prop}[thm]{Proposition}
\newtheorem{exam}[thm]{Example}

\newtheorem{defn}[thm]{Definition}

\theoremstyle{remark}

\makeatletter
\newcommand{\bddots}{%
  \mathinner{\mkern1mu\raise\p@\vbox{\kern7\p@\hbox{.}}\mkern2mu
    \raise4\p@\hbox{.}\mkern2mu\raise7\p@\hbox{.}\mkern1mu}}

\makeatother

\def\sn{{\mathsf n}}

 \def\a{{\alpha}}
 \def\b{{\beta}}
 \def\g{{\gamma}}
 \def\k{{\kappa}}
 \def\t{{\theta}}
 \def\l{{\lambda}}
 
 \def\o{{\omega}}
 \def\s{\sigma}
 \def\la{{\langle}}
 \def\ra{{\rangle}}

 \def\CD{{\mathcal D}}
 \def\CF{{\mathcal F}}
 \def\CH{{\mathcal H}}

 \def\CL{{\mathcal L}}
 
 \def\CP{{\mathcal P}}

 \def\CU{{\mathcal U}}
 \def\CV{{\mathcal V}}
 \def\CA{{\mathcal A}}
 
 \def\BB{{\mathbb B}}

 \def\NN{{\mathbb N}}

 \def\RR{{\mathbb R}}
  \def\SS{{\mathbb S}}
 
      \def\proj{\operatorname{proj}}
      
      \def\vi{\varphi}

\newcommand{\wh}{\widehat}
\def\rhow{{\lceil\frac{s}{2}\rceil}}
\newcommand{\rhoh}{{\lfloor\frac{s}{2}\rfloor}}

\def\ball{\mathbb{B}^d}
\def\sph{\mathbb{S}^{d-1}}
\def\f{\frac}

\def\ln{\|\kern-{}}

\graphicspath{{./}}
\begin{document}

\title[Spectral approximation on the unit ball]
{Spectral approximation on the unit ball}

\author{Huiyuan Li}
\address{Institute of Software
Chinese Academy of Sciences\\ Beijing 100190, China}
\email{huiyuan@iscas.ac.cn}
\author{Yuan Xu}
\address{  Department of Mathematics\\ University of Oregon\\
    Eugene, Oregon 97403-1222.}
\email{yuan@math.uoregon.edu}

\date{\today}
\keywords{Spectral approximation, Sobolev space, order of approximation, unit ball}
\subjclass{ }

\thanks{The first author was supported by NSFC Grant 91130014. The second author was supported in part by NSF Grant DMS-1106113}

\begin{abstract}
Spectral approximation by polynomials on the unit ball is studied in the frame of the Sobolev spaces
$W^{s}_p(\ball)$, $1<p<\infty$. The main results give sharp estimates on the order of approximation
by polynomials in the Sobolev spaces and explicit construction of approximating polynomials. One 
major effort lies in understanding the structure of orthogonal polynomials with respect to an inner product
of the Sobolev space $W_2^s(\ball)$.  As an application, a direct and efficient spectral-Galerkin method 
based on our orthogonal polynomials is proposed for the second and the fourth order elliptic equations 
on the unit ball, its optimal error estimates are explicitly derived for both procedures in the Sobolev spaces 
and, finally, numerical examples are presented to illustrate the theoretic results. 
\end{abstract}

\maketitle


\section{Introduction}
\setcounter{equation}{0}

Spectral methods have been used recently for solving partial differential equations on the unit disk, unit ball, or other
domains with cylindrical or spherical geometry. Their increasing popularity on these domains lies partially in various 
applications in earth sciences, disk or sphere shaped mirrors and lenses,  fluid flow in a pipe or rotating cylinder, 
accretion disks in planetary astronomy, to name a few. 

In \cite{BM}, Poisson equation on an axisymmetric domain is transformed into a system of two-dimensional  problems by the polar transformation, and the axisymmetric problems are then approximated by an appropriate 
spectral-Galerkin method. Fast spectral-Galerkin methods for Helmholtz equations on a disk or a cylinder are proposed in \cite{Shen97,Shen2000}, using the polar transformation with essential pole conditions  and the Chebyshev or Legendre polynomial bases
in the radial  direction.  Subsequently, these types of spectral--Galerkin methods have been extended  to other domains with spherical geometries, including the 3-dimensional ball \cite{AQ07,Shen99,ShenWang}.
Meanwhile, mixed Jacobi-Fourier spectral method are presented for elliptic equations on a disk  \cite{Matsu,WangGuo}
and a mixed Jacobi-harmonic spectral approximation is proposed in \cite{GuoHuang} for a Navier-Stokes equation 
in a ball.  See \cite{Boyd2001, Boyd, CHQZ}   for a comprehensive review 
of  spectral methods and  their special treatments  in polar/spherical coordinates.
Moreover, an alternative approach for solving differential equations in a smooth domain
is to map the domain into the unit ball and then apply a spectral method \cite{ACH1,ACH2, AH}.

One of the challenging problems in the spectral methods on the unit ball is to measure and estimate the errors of approximation in 
genuine, instead of anisotropic, Sobolev norms. Such estimates were established for the product domain in 
\cite{CHQZ,CQ} but has been lacking in most of the works on the unit disk or the unit ball. The problem of characterizing 
best approximation by the smoothness 
of functions is intensively studied in approximation theory. The two problems are closely related but not exactly
the same as we shall explain below. The purpose of this paper is to conduct a comprehensive study for the spectral 
approximation on the unit ball $\BB^d$ of $\RR^d$, making use of recent advances in both approximation theory,
orthogonal polynomials, and spectral methods. 

Spectral approximation for solving an elliptic equation on $\BB^d$ looks for approximate solutions that are 
polynomials written in terms of certain orthogonal basis on the ball and their coefficients are determined by 
the Galerkin method. To understand the convergence of such an approximation process, it is necessary to 
study polynomial approximation in the Sobolev space $W_2^s(\ball)$, where $s$ is a positive integer,  that 
consists of functions whose derivatives up to $s$-th order are all in $L^2(\ball)$. In some literatures, the space 
$W_2^s(\ball)$ is called $H^s(\ball)$.  For $f \in W^s_2(\ball)$, let $S_n^{-s}  f$ denote its best polynomial approximation
of degree at most $n$. For spectral approximation, the desired estimate for $f \in W_2^r(\ball),\, r\ge s$ is of the form 
\begin{equation} \label{spec-approx}
    \|f - S_n^{-s} f \|_{W_2^k(\ball)} \le c\,n^{-r + k} \|f\|_{W_2^r(\ball)}, \quad 0 \le k \le s,
\end{equation}
where $\|\cdot\|_{W_2^r(\ball)}$ denotes the norm of $W_2^r(\ball)$ defined by 
$$
  \|f\|_{W_2^r(\ball)}: =  \Big (\sum_{|\a| \le r} \| \partial^\a f\|_{L^2(\ball)}\Big)^{1/2}. 
$$
One of the main result of this paper is to establish this estimate and, more generally, establish its analogue 
in the space $W_p^s(\ball)$ for $1< p < \infty$. 

The difficulty of quantifying the error of polynomial approximation on the unit ball lies in the strong 
influence of the boundary of the ball on the approximation behavior. This is well documented for 
approximation on a closed interval on the real line. A complete characterization of best approximation 
on the unit ball is only carried out recently. In \cite{DX10}, two moduli of smoothness and their equivalent
$K$-functionals were introduced and used to establish both direct and inverse theorems that characterize
the behavior of best approximation on the unit ball. In \cite{DX11}, approximation in the Sobolev space
was studied and estimate \eqref{spec-approx} for $s= 0$ was established, more generally for 
$1 \le p \le \infty$ (\cite[Corollary 5.4]{DX11}), and the derivative estimates were established for
angular derivatives, which however do not imply \eqref{spec-approx}. What we can prove relatively
effortless (see Theorem \ref{eq:approx-weight} below) is the following estimate 
\begin{equation} \label{spec-approx2}
    \left \| \phi^{\frac{|\a|}{2} } (\partial^\a  f - \partial^\a  S_n^0 f) \right \|_{L^2(\ball)} \le  c\, n^{-s} \|f\|_{W_2^s(\ball)},\quad
        |\a| \le s-1, \quad \a \in \NN_0^d, 
\end{equation}
where $\phi(x) = 1 -\|x\|^2$ vanishes on the boundary sphere $\sph$ of $\ball$ and $S_n^0$ is the partial
sum of the Fourier orthogonal expansion in $L^2(\ball)$. This estimate, however, is weaker than 
\eqref{spec-approx} because of the power of $\phi(x)$ in its left hand side.

It tuns out that what we need for proving \eqref{spec-approx} is the orthogonal structure of the Soblolev space 
$W_2^s(\ball)$, not the orthogonal structure of $L^2(\ball)$. An essential step in our study is to study orthogonal 
polynomials for with respect to the inner product 
$$
\la f ,g \ra_{-s}:= \la \nabla^s f, \nabla^s g \ra_{\ball} + \sum_{k=0}^{\rhow -1} \la \Delta^k f, \Delta^k g \ra_{\sph} 
$$
of  $W_2^s(\ball)$, which we call the Sobolev orthogonal polynomials. 
Initially motivated by direct and efficient spectral method of Atkinson and his collaborators that uses orthogonal 
polynomials to solve linear elliptic equations on the disk \cite{ACH1, ACH2, AH},  the Sobolev orthogonal 
polynomials on the ball with respect to $\la \cdot, \cdot \ra_{-1}$ were studied in \cite{X08} and those with 
respect to $\la \cdot, \cdot \ra_{-2}$ were 
studied in \cite{PX, X06}. In these works, Sobolev orthogonal bases were  constructed in terms of the orthogonal polynomials 
for $L^2(\varpi_\mu,\ball)$ 
with $\mu=1$ and $2$, respectively, where  the weight  function $\varpi_\mu(x) := (1-\|x\|^2)^{\mu}$, 
which are given explicitly in terms of spherical harmonics and the Jacobi polynomials $P_n^{(\a,\b)}(t)$ 
that are orthogonal polynomials with respect to $(1-t)^\a (1+t)^\b$ on $[-1,1]$. For larger $s$, however, the 
orthogonal structure is more complicated, and we need to extend the orthogonal basis for $L^2(\varpi_\mu,\ball)$
to allow $\mu$ to be negative integers, which in turn requires us to use extensions of the Jacobi polynomials 
with negative indexes. This is prompted by the realization that the Sobolev orthogonal polynomials for 
$s =-1$ and $s= -2$ in \cite{X06, X08} can be expressed in terms of orthogonal polynomials for $L^2(\varpi_\mu,\ball)$
with $\mu = -1$ and $-2$, and, heuristically, the negative weight could cancel out the $\phi^{|\a|/2}$ term
in \eqref{spec-approx2}. The Jacobi polynomials with negative indexes have been used in spectral 
approximation on other domains in \cite{GSW06,GSW08,LiShen,SWL07}. One of our main results is 
an explicitly constructed mutually orthogonal polynomial basis for $\la \cdot,\cdot\ra_{-s}$, which could be 
used as the building blocks for the spectral-Galerkin  method. 

For $f \in L^2(\ball)$, its $n$-th best polynomial approximation is given by the $n$-th partial sum of its 
Fourier orthogonal expansion on the ball. For $f \in W_2^s(\ball)$, we shall prove that the best approximating
polynomials to $f$ are $S_n^{-s} f$, the partial sums of the Fourier orthogonal expansion in 
$W_2^r(\ball)$ equipped with the inner product $\la \cdot,\cdot \ra_{-s}$, which can be expressed explicitly
in terms of the mutually orthogonal polynomials that we constructed. For the $W_p^s(\ball)$ with $p \ne 2$, 
the best approximating 
polynomial is not explicitly known, but we are able to show that a near-best approximating polynomial,
denoted by $S_{n,\eta}^{-s}$ and defined via a smooth cut-off function $\eta$, satisfies our sharp estimate in
$W_p^s(\ball)$. Both $S_n^{-s} f$ and $S_{n,\eta}^{-s} f$ are given by explicit formulas that can be easily 
computed numerically (see Section 4). Our main result on approximation in the Sobolev space is the following:  

\begin{thm}
Let $s,r = 1,2\ldots$. For any $f \in W_p^r(\ball)$, $r\ge s$, $1 < p < \infty$, there is a constant $c$ independent of
$f$ and $n$, such that 
$$
   \|f -S_{n,\eta}^{-s} f\|_{W_p^k(\ball)} \le c n^{-r + k} \|f\|_{W_p^r(\ball)}, \qquad  k =0,1, \ldots, s,
$$
where $S_{n,\eta}^{-s} f$ can be taken as $S_{n}^{-s}f $ for $p=2$. 
\end{thm}

More precise results of this nature are stated in Section 4.1 below. To illustrate the application of this 
result in the spectral approximation, we will consider two examples, the Helmholtz equation and the biharmonic 
equation on the unit ball, and demonstrate how our results on approximation in the Sobolev space can be 
used to error estimates in the  spectral-Galerkin method. Furthermore, we provide numerical examples for
these equations for $d=2$ and $d=3$, which further illustrate our findings.

The paper is written with readers in both approximation theory community and spectral method community 
in mind. The problem of \eqref{spec-approx} is originated and studied in the spectral method, which is 
closely tied to the problem of characterizing best approximation that has been a central theme and studied 
intensely in approximation theory. Our approach uses a mixed bag of tools, developed in both approximation 
theory and spectral methods. It is our hope that this paper will stimulate further collaboration between the two 
communities. 

The paper is organized as follows. In the next section we present background materials, orthogonal polynomials 
on the unit ball, Fourier orthogonal expansions, and recent results on approximation on the unit ball. The 
orthogonal structure of the Sobolev space is developed in Section 3. The main results on approximation by
polynomials in the Sobolev space are stated and proved in Section 4. Finally, in Section 5, we discuss 
applications of our main results in the spectral-Galerkin  methods and present our numerical examples. 
To keep the presentation fluent, we leave technical details of extending orthogonal bases to negative
indexes and proving equivalence of norms in the Sobolev space to Appendix A and Appendix B, respectively. 


\section{Preliminary and background}
\setcounter{equation}{0}

For $x,y  \in \RR^d$, we use the usual notation of $\|x\|$ and $\la x,y \ra$ to denote the Euclidean 
norm of $x$ and the dot product of $x,y$.  The unit ball and the unit sphere in $\mathbb{R}^d$ are denoted, respectively, by
$$
\ball :=\{x\in \mathbb{R}^d: \|x\| \le 1\} \quad \textrm{and} \quad \SS^{d-1}:=\{\xi\in \mathbb{R}^d: \|\xi\| = 1\}.
$$
Throughout this paper, we let $\partial_i$ denote the $i$-th partial derivative operator, let 
$\nabla = (\nabla_1,\ldots, \nabla_d)$ be the gradient and let  $\Delta = \partial_1^2 + \ldots + \partial_d^2$ be
the usual Laplace operator. We denote by $c$ a constant that depends only on $p$, $s$ or other fixed parameters,
its value may change from line to line. 

\subsection{Spherical harmonics} 
We follow the notation in \cite{DaiX}. Let $\CP_n^d$ denote the space of homogeneous polynomials of degree 
$n$ in $d$ variables. It is known that 
$$
      \dim \CP_n^d  = \binom{n+d-1}{n}. 
$$
Harmonic polynomials of $d$-variables are polynomials in $\CP_n^d$ that satisfy the Laplace equation 
$\Delta Y = 0$. Spherical harmonics are the restriction of harmonic polynomials on the unit sphere. 
Let $\mathcal{H}_n^d$ denote the space of spherical harmonic polynomials of degree $n$. It is well--known 
that
$$
         a_n^d: = \dim \mathcal{H}_n^d = \binom{n+d-1}{n} - \binom{n+d-3}{n-2}.
$$
If $Y \in \mathcal{H}_n^d$, then $Y(x) = \rho^n Y(\xi)$ in spherical--polar coordinates $x = \rho \xi$. We call 
$Y(x)$ a solid spherical harmonic. Evidently, $Y$ is uniquely determined by its restriction on the sphere.
We shall also use $\CH_n^d$ to denote the space of solid spherical harmonics. 

The spherical harmonics of different degrees are orthogonal with respect to the inner product
$$
        \la f, g \ra_{\sph}: =  \frac{1}{\o_d}\int_{\sph} f(\xi) g(\xi) d\s(\xi),
$$
where $d \o$ is the surface measure and $\o_{d}={2\pi^{\frac{d}{2}}}/{\Gamma(\frac{d}{2})}$ is the surface
area; the inner product is normalized so that $\la 1,1 \ra_{\sph}=1$.

In spherical polar coordinates, the Laplace operator can be written as 
\begin{equation} \label{eq:Delta}
   \Delta = \frac{d^2}{d\rho^2} + \frac{d-1}{\rho} \frac{d}{d\rho} + \frac{1}{\rho^2} \Delta_0,
\end{equation}
where $\rho = \|x\|$ and $\Delta_0$, the spherical part of $\Delta$, is the Laplace-Beltrami operator that has spherical
harmonics as eigenfunctions; more precisely, for $n = 0,1,2,\ldots$, 
\begin{equation} \label{eq:LaplaceBeltrami}
        \Delta_0  Y = - n(n+d-2) Y, \qquad Y \in \CH_n^d. 
\end{equation}

Let $f \in L^2(\sph)$ and let $\{Y_\ell^n: 1 \le \ell \le a_n^d\}$ be an orthonormal basis of $\CH_n^d$ such that
$\la Y_\ell^n, Y_\iota^n\ra_{\sph}=\delta_{\ell,\iota}$.
The spherical harmonic expansion of $f$ is defined by 
$$
   f(\xi) = \sum_{n=0}^\infty \sum_{\ell=1}^{a_n^d} \wh f_\ell^n Y_\ell^n(\xi), 
      \qquad \wh f_\ell^n = \la f, Y_\ell^n \ra_{\sph}.  
$$
We define the partial sum of the harmonic expansion and the projection operator 
$ \proj_n^\CH : L^2(\sph) \mapsto \CH_n^d$ by 
\begin{equation}\label{eq:SnH}
  S_n^\CH f(\xi) : = \sum_{m=0}^n \proj_m^\CH f (\xi) \quad\hbox{and}\quad  
     \proj_m^\CH f (\xi) : = \sum_{\ell=1}^{a_m^d} \wh f_\ell^m Y_\ell^m(\xi),
\end{equation}
respectively. The projection operator is independent of the choice of orthonormal basis of $\CH_n^d$. 
Furthermore, since $\proj_m^\CH f$ is homogeneous, we can extend its definition to the unit ball by 
$\proj_m^\CH f (x) = \rho^m \proj_m^\CH f (\xi)$ for $x = \rho \xi \in \ball$. We extend $S_n^\CH f$ accordingly. If 
$h$ is a harmonic function on the unit ball, then $S_n^\CH h$ is the best approximation to $h$ in $L^2(\ball)$. 

\subsection{Orthogonal structure on the unit ball}
Our basic reference in this section is \cite{DX}. For $\mu \in \RR$, let $\varpi_{\mu}$ be the weight function defined by
$$
    \varpi_{\mu}(x) = (1-\|x\|^2)^\mu, \qquad  \|x\| < 1.
$$
The classical orthogonal polynomials on the unit ball are orthogonal with respect to the inner product
\begin{equation}\label{ball-ip}
   \la f,g \ra_\mu = \frac{1}{b_{d}^{\mu}}\int_{{\ball}}\, f(x)\, g(x) \, \varpi_{\mu}(x) \, dx,
   \quad \mu>-1.
\end{equation}
where $b_{d}^{\mu}=\frac{\pi^{\frac{d}{2}}\Gamma(\mu+1)}{\Gamma(\mu+\frac{d}{2}+1)}$ is the normalization constant such that 
$\la 1,1 \ra_\mu=1$. For clarity,  we write $\la f, g\ra_{\ball}=\la f, g\ra_{0}$.

Let $\Pi^d$ denote the space of polynomials in $d$ real variables. For $n = 0,1,2,\ldots,$ let $\Pi_n^d$ denote
the linear space of polynomials in $d$ variables of (total) degree at most $n$. A polynomial $P \in \Pi_n^d$ is 
called orthogonal with respect to $\varpi_{\mu}$ on the ball if $\la P, Q\ra_\mu =0$ for all $Q \in \Pi_{n-1}^d$. Let 
$\CV_n^d(\varpi_{\mu})$ denote the space of orthogonal polynomials of total
degree $n$ with respect to $\varpi_{\mu}$. It is well--known that 
$$
  \dim \Pi_n^d = \binom{n+d}{n} \quad\hbox{and} \quad \dim \CV_n^d(\varpi_{\mu}) = \binom{n+d-1}{n}.
$$ 
The space of $\CV_n^d$ has many different bases. Let $\{P^n_{\alpha}(x) : |\alpha|=n\}$ denote a basis of 
$\CV_n^d(\varpi_{\mu})$, then $\la P_\a^n, P_\b^m \ra_{\mu} =0$ if $n \ne m$. The basis is called mutually
orthogonal if $\la P_\a^n, P_\b^n \ra_{\mu} = 0$ whenever $\a \ne \b$, and it is called orthonormal if 
$\la P_\a^n, P_\a^n \ra_{\mu} =1$ in addition. Let $(a)_n:= a(a+1) \ldots (a+n-1)$ be the Pochhammer symbol.
We use the standard  multi--index notation that, for $\a \in \NN_0^d$, 
$$
\a! = \a_1!\cdots \a_d!, \quad \hbox{and} \quad (\a)_\g = (\a_1)_{\g_1}\cdots (\a_d)_{\g_d}.
$$

One basis of $\CV_n^d(\varpi_{\mu})$ is given in terms of the Jacobi polynomials and spherical  harmonics. 
Let $P_j^{(\mu,\nu)}(t)$ denote the usual Jacobi orthogonal polynomial of degree $j$
with respect to weight function $(1-t)^{\mu}(1+t)^{\nu}$  on $[-1,1]$.

\begin{prop}
For $n \in \NN_0$ and $0 \le j \le \tfrac{n}{2}$, let $\{Y_\ell^{n-2j}: 1\le \ell\le a_{n-2j}^d\}$ be an 
orthonormal basis for $\mathcal{H}_{n-2j}^d$. Define 
\begin{equation}\label{baseP}
P_{j,\ell}^{\mu,n}(x) := 
\frac{(n-j+\tfrac{d}{2})_j}{(n-j+\tfrac{d}{2}+\mu)_j}  P_{j}^{(\mu, n-2j + \frac{d-2}{2})}(2\,\|x\|^2 -1)\, Y_\ell^{n-2j}(x).\end{equation}
Then the set $\{P_{j,\ell}^{\mu,n}(x): 0 \le j \le \tfrac{n}{2}, \,1 \le \ell \le a_{n-2j}^d \}$ is a mutually
orthogonal basis of $\CV_n^d(\varpi_{\mu})$ whenever $\mu>-1$. More precisely, 
$$
\la P_{j,\ell}^{\mu,n}, P_{k,\iota}^{\mu,m}\ra_\mu =  h_{j,n}^{\mu}  \delta_{n,m}\,\delta_{j,k}\,\delta_{\ell,\iota},
$$
where $ h_{j,n}^{\mu}$ is given by 
\begin{equation} \label{eq:Hjn-mu}
 h_{j,n}^{\mu}: = 
 \frac{ (\mu +1)_j   (1-n-\tfrac{d}{2})_j (\frac{d}2)_n  } {j! (1-n-\tfrac{d}{2}-\mu)_j (\tfrac{d}{2}+\mu+1)_n  }.
\end{equation} 
\end{prop}

This is a standard mutually orthogonal basis on the unit ball; see \cite[p. 39]{DX}. We include a
constant in the definition of $P_{j,\ell}^{\mu,n}(x)$ in order to extend this definition to the case
of $\mu \le -1 $, which is explained in Appendix \ref{Gball}. 

It is known that orthogonal polynomials with respect to $\varpi_{\mu}$ are eigenfunctions of a second order
differential operator $\CD_\mu$. More precisely, we have
\begin{align} \label{eq:Bdiff}
    \CD_\mu P =  -(n+d) (n + 2 \mu)P, \qquad \forall P \in \CV_n^d(\varpi_{\mu}), 
\end{align} 
where 
\begin{equation*}
  \CD_\mu := \Delta  - \sum_{j=1}^d \frac{\partial}{\partial x_j} x_j \left[
  2 \mu  + \sum_{i=1}^d x_i \frac{\partial  }
  {\partial x_i} \right].
\end{equation*}

In term of the mutually orthogonal basis $\{P_{j,\ell}^{\mu,n} :1 \le \ell \le a_{n-2j}^d, 0 \le j \le n/2, n =0,1,\ldots \}$,
the Fourier orthogonal expansion of $f \in L^2(\varpi_{\mu}, \BB^d)$ is defined by 
$$
   f(x) = \sum_{n=0}^\infty\sum_{0\le j \le n/2} \sum_{\ell=1}^{a_{n-2j}^d} \wh f_{j,\ell}^n P_{j,\ell}^{\mu,n}, 
         \quad\hbox{where}\quad  \wh f_{j,\ell}^n : =\frac{1}{h_{j,n}^\mu} \la f, P_{k,\ell}^{\mu,n} \ra_{\mu}.  
$$
We define the partial sum of the orthogonal expansion and the projection operator $\proj_n^\mu: 
L^2(\varpi_{\mu}, \BB^d) \mapsto \CV_n^d(\varpi_\mu)$ by
\begin{equation}\label{eq:Sn-mu}
  S_n^\mu f(x) : = \sum_{m=0}^n \proj_m^\mu f (x) \quad\hbox{and}\quad  
     \proj_m^\mu f (x) : =\sum_{0\le j \le m/2}   \sum_{\ell =1}^{a_{m-2j}^d}  \wh f_{j,\ell}^m P_{j,\ell}^{\mu,m},
\end{equation}
respectively. By definition, $S_n^\mu $ is the orthogonal projection of $L^2(\varpi_{\mu}, \ball)$ onto $\Pi_n^d$; 
that is, $S_n^\mu f = f$ if $f \in \Pi_n^d$ and 
\begin{align*}
   \la S_n^\mu f- f, v\ra_{\mu} = 0, \ \  \forall v\in \Pi_n^d.
\end{align*}

\subsection{Fourier orthogonal expansions and approximation}

For $1 \le p < \infty$, let $\|f\|_{p, \sph}$ denote the $L^p(\sph)$ norm 
$$
  \| f \|_{p,\sph} : = \left( \frac1{\o_d} \int_{\sph} |f(\xi)|^p  d\s(\xi) \right)^{1/p},  
$$
and let $\|f\|_{\infty, \sph} = \|f\|_{\infty}$ be the uniform norm for $f \in C(\sph)$. Furthermore, 
for $1 \le p < \infty$, let $\|f\|_{p,\mu}$ denote the $L^p(\varpi_{\mu},\BB^d)$ norm 
$$
  \| f \|_{p,\mu} : = \left( \frac1{b^\mu_d} \int_{\ball} |f(x)|^p \varpi_{\mu}(x) dx \right)^{1/p},  
$$
and let $\|f\|_{\infty, \mu} = \|f\|_{\infty}$ be the uniform norm for $f \in C(\ball)$. In the case of $\mu =0$,
we shall denote the norm by $\|f\|_{p,\ball}: = \|f\|_{p,0}$. 

In the remaining of this subsection, we write $L^p$  (resp. $\|f\|_p$) for either $ L^p(\varpi_{\mu}, \ball)$ or $L^p(\sph)$ (resp. $\|f\|_{p,\mu}$ or $\|f\|_{p,\sph}$), and 
write $\proj_n f$ and $S_n f$ for either $\proj_n^\CH f$ and $S_n^\CH f$ defined in \eqref{eq:SnH} or 
$\proj_n^\mu f$ and $S_n^\mu f$ defined in \eqref{eq:Sn-mu}. When the setting is on $\sph$, 
$\Pi_n^d = \Pi_n^d(\sph)$.

\begin{defn} \label{defn:BestApp}
Let $f \in L^p$ if $1 \le p < \infty$ and $f\in C$ if $p = \infty$. For $n \ge 0$, the error of the
best approximation to $f$ by polynomials of degree at most $n$ is defined by 
\begin{equation}\label{bestapp}
    E_n (f): = \inf_{ g  \in \Pi_n^d } \| f - g \|_{p}, \qquad 1 \le p \le \infty.
\end{equation}
With norm specified, we shall write $E_n(f)_{p,\mu}$, $E_n(f)_{p,\ball}$ and $E_n(f)_{p,\sph}$.
\end{defn}
 
The standard Hilbert space theory shows 
that $S_n f$ is the best $L^2$ approximation to $f$; that is, 
$$
       E_n (f)_{2} = \|f - S_n^\mu f \|_{2}
$$ 
For $p \ne 2$, we no longer know the polynomial of best approximation explicitly, but a near best approximation 
is known (see, for example, \cite[p. 284]{DaiX}). 

\begin{defn}\label{defn:Lnf}
A $C^\infty$-function $\eta$ on $[0,\infty)$ is called   an admissible cut-off function if $\eta(t)=1$ for  
$0\le t\le 1$ 
and $\eta(t)=0$ for $t \ge 2$. If $\eta$ is such a function, define
\begin{equation}\label{Vnf}
 S_{n,\eta} f(x):= \sum_{k=0}^{\infty} \eta\left(\f k{n}\right) \proj_n f(x). 
\end{equation}
When $\proj_n$ is specified, we will write $S_{n,\eta}^\CH f$ and $S_{n,\eta}^\mu f$ accordingly.  
\end{defn}

Since $\eta$ is supported on $[0,2)$, the summation in $ S_{n,\eta} f$ can be terminated at $k =2 n-1$,
so that $ S_{n,\eta} f$ is a polynomial of degree at most $2n-1$. It approximates $f$ as well as the best
approximation polynomial of degree $n$. 

\begin{thm} \label{thm:near-best}
Let $f \in L^p$ if $1 \le p < \infty$ and $f \in C$ if $p = \infty$. Then
\begin{enumerate}[  \quad \rm(1)]
 \item $ S_{n,\eta}  f \in \Pi_{2n-1}^d$ and $ S_{n,\eta} f = f $ for $f \in \Pi_n^d$.
 \item For $n \in \NN$, $\| S_{n,\eta}  f\|_{p} \le c \|f\|_{p}$.
 \item For $n \in \NN$, there is a constant $c > 0$, independent of $f$, such that 
\begin{equation}\label{eq:near-best}
        \|f -  S_{n,\eta}  f\|_{p} \le (1+c) E_{n}(f)_{p}.
\end{equation}
\end{enumerate}
\end{thm}

The quantity $E_n(f)_{p,\sph}$ and $E_n(f)_{p,\mu}$ can be characterized by the smoothness of the function $f$; 
see \cite{DX10, DX11} and Section 4 below. 

\subsection{Best approximation on the unit sphere}

We recall result on the characterization of best approximation by polynomials in $L^p(\sph)$ in terms of the 
smoothness of the functions. In approximation theory, smoothness of 
a function is usually measured by the modulus of smoothness and its equivalent $K$-functional. 
Since we are primarily interested in functions in Sobolev spaces, we shall state the result only in 
terms of $K$-functional. 

For $s = 0, 1,2, \ldots$ and $1 \le p < \infty$, we define the Sobolev space $W_p^s(\sph)$ to be 
the space of functions whose spherical/angular derivatives up to $s$-th order are all in $L^p(\sph)$. For $p =\infty$, 
we replace $L^p$ space by the space $C(\sph)$ of continuous functions on $\sph$.  The norm 
and semi-norm of  $W_p^s(\sph)$ can be defined by  
\begin{align}\label{Snom1S}
   &\|f\|_{W_p^s(\sph)} : =  \|f\|_{p,\sph}  + |f|_{W_p^s(\sph)}^\circ,
\quad |f|_{W_p^s(\sph)}^\circ: = \sum_{1 \le i<j \le d} \|D_{i,j}^s g\|_{p,\sph},
\end{align}
where $D_{i,j}:= x_i \partial_j - x_j \partial_i,\ 1 \le i< j \le d$  are angular differential operators. 
In polar coordinates on the plane $(x_i,x_j) = r_{i,j} (\cos \t_{i,j}, \sin \t_{i,j})$, 
$D_{i,j} = \frac{\partial}{\partial \t_{i,j}}$, which explains their name; see \cite[Section 1.8]{DaiX}
for further properties of these operators.

\begin{defn}  \label{def:K-funcS}
Let  $f\in  L^p(\sph)$ if $1\le p<\infty$ and $f\in C(\sph)$ if $p=\infty$.  For $s\in \NN_0$ an $t\ge 0$, define the K-functional 
\begin{equation} \label{eq:K-func-sphere}
   K_{s}(f,t)_{p, \sph} : = \inf_{g\in W^s_p(\sph)} \left\{ \|f-g\|_{p,\sph} + t^s | g|_{W_p^{s}(\sph)}^{\circ} \right\}. 
\end{equation}
\end{defn}

This definition and the characterization of best approximation below were established in \cite{DX10}, where an
equivalent modulus of smoothness was also defined. 

\begin{thm} \label{thm:JacksonS}
Let $s \in \NN$ and let $f \in L^p(\sph)$ if $1 \le p < \infty$, and $f\in C(\sph)$ if $p=\infty$. Then
\begin{equation}\label{JacksonS}
   E_n (f)_{p,\sph} \le c\, K_s(f,n^{-1})_{p,\sph} 
\end{equation}
and
\begin{equation} \label{inverseS}
 K_{s} (f,n^{-1})_{p,\sph} \leq  c\, n^{-s} \sum_{k=1}^n k^{{s}-1} E_{k}(f)_{p,\sph}.
\end{equation}
\end{thm}

The estimate \eqref{JacksonS} is usually called direct, or Jackson, inequality, while \eqref{inverseS} is usually
called inverse inequality. If $f \in W_p^{s}(\sph)$, then we can choose $g = f $ in the infimum of 
$K$-functional, which gives the following corollary. 

\begin{cor} \label{cor:JacksonS}
Let ${s}\in \NN$ and let $f \in W_p^{s}(\sph)$ if $1 \le p < \infty$, and $f\in C(\sph)$ if $p=\infty$. Then
\begin{equation}\label{JacksonS2}
   E_n (f)_{p,\sph}  \le c\, n^{-s}   |f|_{W_p^s(\sph)}^\circ \le c\, n^{-s}   \|f\|_{W_p^s(\sph)}.    
\end{equation}
\end{cor}

We will also need an estimate in the fractional order Sobolev space $W^{r+\t}_p(\sph)$, where $0 < \t < 1$, 
which is defined as the interpolation space $(W^{s}_p(\sph), W^{s+1}_p(\sph))_{\theta,p}$; see Appendix B. 

\begin{thm}
\label{th:approxHarm}
If $f\in W^{r+\t}_p(\sph)$  for  $r\in \NN_0$, $0\le \t <1$ and $1<p<\infty$, then
\begin{align}
\label{eq:approxHarm}
\|f-S_{n,\eta}^{\CH} f\|_{W^{s+\t}_p(\sph)} \le c n^{-r+s} \|f\|_{W^{r+\t}_p(\sph)}, 
\quad  s=0,1,\dots,r.
\end{align}
\end{thm}

\begin{proof}
Since $D_{i,j}$ maps $\CH_n^d$ to itself  for $1\le i<j \le d$  \cite[Lemma\,1.8.3]{DaiX}, it follows readily that 
$D_{i,j}\proj_n^{\CH}=\proj_n^{\CH}D_{i,j}$. As a result, $ D^s_{i,j} S^{\CH}_{n}f  = S^{\CH}_{n} D^s_{i,j} f $ 
and $ D^s_{i,j} S^{\CH}_{n,\eta}f  = S^{\CH}_{n,\eta} D^s_{i,j} f $.
Thus by \eqref{eq:near-best} and \eqref{JacksonS2},
\begin{align*}
 &\big\| D^s_{i,j} (S^{\CH}_{n,\eta}f- f) \big\|_{p,\sph} 
 \le c E_n(D^s_{i,j}  f)_{p,\sph} \le c n^{s-r} | D^s_{i,j}   f|_{W^{r-s}_p(\sph)}^{\circ},
\end{align*}
for $ s =0,1,\dots,r$, which gives \eqref{eq:approxHarm} for $\t =0$. Consequently, it follows that 
\begin{align*}
 &\big\| S^{\CH}_{n,\eta}f- f \big\|_{W^s_p(\sph)} 
\le  c n^{s-r} \|  f\|_{W^{r}_p(\sph)},
\end{align*}
which implies that  $\| S^{\CH}_{n,\eta}-I \| _{\CL(W^{r}_p(\sph), W^{s}_p(\sph))} \le c n^{s-r}  $ 
for any $r\ge s$,  where $\|\cdot\|_{\CL(X,Y)}$ denotes the norm of the operator from $X\mapsto Y$. It 
then follows from \eqref{eq:theta} that
\begin{align*}
 \|S^{\CH}_{n,\eta}-I \|_{\CL(W^{r+\theta}_p(\sph), W^{s+\theta}_p(\sph))  }& 
   \le c n^{s-r}.
\end{align*}
This completes the proof of \eqref{eq:approxHarm}.
\end{proof}

\subsection{Best approximation on the unit ball}
We recall result on best approximation by polynomials in $L^p(\ball)$.
We define the Sobolev space $W_p^s(\varpi_\mu,\BB^d)$ to be the space of functions whose 
derivatives up to the $s$-th order are all in $L^p(\varpi_\mu,\ball)$. For $p =\infty$, we replace $L^p$ space
by the space $C(\BB^d)$ of continuous functions on $\BB^d$. The norm of  $W_p^s(\varpi_\mu, \BB^d)$ 
is defined by
\begin{equation}\label{Snom1B}
   \|f\|_{W_p^s(\varpi_\mu, \BB^d)} : = \Big( \sum_{ |\a| \le s} \|\partial^\a f\|_{p,\mu}^p \Big)^{1/p}.
\end{equation}
When $\mu = 0$, we write $\|f\|_{W_p^s(\BB^d)} : = \|f\|_{W_p^s(\varpi_0, \BB^d)}$. 

A $K$-functional on the unit ball (and its equivalent modulus of smoothness) is 
defined in \cite{DX10} and used to characterize the best approximation in $L^p(\varpi_\mu, \ball)$. 
Throughout this paper, we define
$$  
       \varphi(x): = \sqrt{1-\|x\|^2}. 
$$

\begin{defn} \label{defn:K-funcB2}
Let $f \in L^p(\varpi_{\mu},\ball)$ if $1 \le p < \infty$ and $f \in
C(\ball)$ if $p = \infty$.  For ${s} \in \NN$ and $t>0$, define
\begin{align*}  
   K_{{s},\vi}(f,t)_{p,\mu}
   := \inf_{g\in W_p^{s}(\varpi_{\mu}, \ball)} \Big\{ \|f-g\|_{p,\mu} 
       & + t^{s}  |g|_{W_p^s(\varpi_\mu, \ball)}^\circ \Big \}.
\end{align*}
where 
\begin{equation} \label{semi-normB}
   |g|_{W_p^s(\varpi_\mu, \ball)}^\circ: = \sum_{1 \le i<j \le d}  \| D_{i,j}^{s}  g\|_{p,\mu} 
        +   \sum_{i=1}^d  \| \varphi^{s} \partial_i^{s} g\|_{p,\mu}.
\end{equation}
\end{defn}

Both direct and inverse theorems were established for $E_n(f)_{\mu, p}$ in \cite[Theorem 6.6]{DX10}. 
They are analogues of Theorem \ref{thm:JacksonS}. We will only state the corollary that is an analogue
of Corollary \ref{cor:JacksonS} and only for $\mu =0$, where we write $|f|_{W_p^s(\ball)}^\circ =
  |f|_{W_p^s(\varpi_0, \ball)}^\circ$.

\begin{cor}
\label{cor:Jackson}
Let ${s}\in \NN$ and let  $f \in W_p^{s}(\ball)$ if $1 \le p < \infty$, and 
$f\in C(\ball)$ if $p=\infty$. Then
\begin{equation}\label{Jackson}
   E_n (f)_{p,\ball} \le c\, n^{-{s}}   |f|_{W_p^{s}(\ball)}^\circ \le  c\, n^{-{s}}   \|f\|_{W_p^{s}(\ball)}.      
\end{equation}
\end{cor}

It should be mentioned that \cite{DX10} contains another $K$-functional that differs from $K_{{s},\vi} (f,t)_{p,\mu}$ in
its last term, which can also be used to estimate $E_n(f)_{p,\mu}$. Several results on approximation in the Sobolev spaces and Lipschitz spaces were established in \cite{DX11}, which contains, for example, the estimates
$$
  \|D_{i,j}^{s} (f - S_{n,\eta}^\mu f)\|_{p,\mu}   \le c  E_n ( D_{i,j}^{s} f)_{p,\mu} \quad 1 \le i < j \le d. 
$$
For the spectral approximation, however, we are more interested in the derivatives $\partial^\a$ instead of the angular
derivatives. One result in this direction can be derived with the help of the following lemma. 

\begin{lem} \label{lem:diff-Sn-mu}
For $\mu > -1$ and $1 \le i \le d$,
\begin{equation} \label{eq:diff-Sn-mu}
    \partial_i S_n^\mu f = S_{n-1}^{\mu+1} (\partial_i f) \quad \hbox{and}\quad   
      \partial_i S_{n,\eta}^\mu f = S_{n-1,\eta}^{\mu+1} (\partial_i f).
\end{equation}
\end{lem}

\begin{proof}
By the definition of Fourie orthogonal expansion, $f - S_n^\mu f = \sum_{m=n+1} \proj_m^\mu f$ and $\proj_m^\mu f
\in \CV_m^d(\varpi_{\mu})$. By Lemma \ref{lm:DiffV}, $\partial_i \proj_m^\mu f \in \CV_{m-1}^d(\varpi_{\mu+1})$. It follows
that $\la \partial_i (f - S_n^\mu f), P\ra_{\mu+1} = 0$ for all $P \in \Pi_{n-1}^d$. Consequently, $
  S_{n-1}^{\mu+1} (\partial_i f - \partial_i S_n^\mu f) = 0$. Since $S_{n-1}^{\mu+1}$ reproduces polynomials of 
degree at most $n-1$, $S_{n-1}^{\mu+1} ( \partial_i S_n^\mu f) =  \partial_i S_n^\mu f$, which implies that  
$$
 0 = S_{n-1}^{\mu+1} (\partial_i f - \partial_i S_n^\mu f) = S_{n-1}^{\mu+1} (\partial_i f)  -  \partial_i S_n^\mu f.
$$
This proves the first identity in \eqref{eq:diff-Sn-mu}. Since $\proj_n = S_n -S_{n-1}$, it follows that 
$
 \proj_{n-1}^{\mu+1} (\partial_i f)  =  \partial_i \proj_n^\mu f,
$
from which the second identity in  \eqref{eq:diff-Sn-mu} follows immediately. 
\end{proof}

\begin{thm}
If $f \in W_p^{s}(\varpi_{\mu}, \ball)$ for $1 \le p < \infty$, or $f \in C^{s}(\ball)$ for $p = \infty$,  then for $|\a| = {s}$, 
\begin{align}\label{eq:approx-weight}
  \| \partial^\a f - \partial^\a S_{n,\eta}^\mu f\|_{p, \mu+|\a|} \le c E_{n-|\a|} (\partial^\a f)_{p,\mu}
    \le c n^{-s} \|f\|_{W_p^s(\varpi_\mu, \ball)}. 
\end{align}
\end{thm}

\begin{proof}
This follows immediately from  Lemma \ref{lem:diff-Sn-mu}, Theorem \ref{thm:near-best} and
Corollary \ref{cor:Jackson}. 
\end{proof}
 
As explained in the introduction, the estimate  with $\mu = 0$ and $p =2$ is weaker than the desired
estimate \eqref{spec-approx} because of the factor $(1-\|x\|^2)^{|\a|}$ that appears in its left hand side.

\section{Orthogonal structure in the Sobolev space}
\setcounter{equation}{0}

In this section, we consider orthogonal structure in the Sobolev space $W_p^s(\BB^d)$. Let 
\begin{align*}
\nabla^{2m} := \Delta^m \quad\hbox{and}\quad  \nabla^{2m+1} := \nabla \Delta^{m},  \qquad m=1,2,\dots.
\end{align*}

\begin{defn} \label{def:ipd-s}
For $s =1,2,\ldots$, we define a bilinear form on the space $W_2^{s}(\ball)$ by
\begin{align} \label{eq:ipd-s}
 & \la f, g\ra_{-s}:= \la \nabla^{s}f, \nabla^{s} g \ra_{\ball}
  + \sum_{k=0}^{\rhow-1}\lambda_{k} \la  \Delta^{k}f, \Delta^{k}g\ra_{\sph},
\end{align}
where  $\l_{k}, \, k=0,1\dots,\rhow-1$, are positive constants.  
\end{defn}

It is easy to see that this defines an inner product for $W_2^s(\ball)$. We denote the space of orthogonal 
polynomials of degree $n$ with respect to this inner product by $\CV_n^d(\varpi_{-s})$. 
The reason that we use the negative index to denote such an inner product will become clear momentarily. 

For our purpose, we need to extend the definition of orthogonal polynomials $P_{j,\ell}^{\mu,n}$ defined 
in \eqref{baseP} so that $\mu$ can be negative. The extension is carried out in Appendix \ref{Gball}. 
Below we shall use  $P_{j,\ell}^{-s,n}$ for $s = 1,2,\ldots$,  and what we essentially need is the following
lemma proved in  Appendix \ref{Gball}. 

\begin{lem} \label{DeltaP}
Let $\mu \in \RR$,  $s\in \NN$ and $n\in \NN_0$. Then for $ 1\le\ell \le a_{n-2j}^d$,
\begin{align} \label{PN2P}
 P^{-s,n}_{j,\ell}(x)
= \frac{(1-n-\tfrac{d}{2})_j}{(-j)_{s} (1-n-\tfrac{d}{2}+2s)_{j-s}}  (\|x\|^2-1)^{s} 
P^{s,n-2s}_{j-s,\ell}(x),  \quad   s \le j\le \tfrac{n}{2}. 
\end{align}
Furthermore, make the convention $P^{\mu,n}_{j,\ell}(x)=0$ if  $j<0$ or $j>\frac{n}{2}$;
and  define  $j_0= \psi(j)$ if $\psi(j):= s+j-n-d/2+1 \in \{ 1, 2,\dots, j\}$, and $j_0 =0$ otherwise. Then for $0\le j\le \frac{n}{2}$,
\begin{align} \label{LaplaceP}
  &   \Delta^{k} P^{-s,n}_{j,\ell} (x) = 4^{k} 
       (n+\tfrac{d}2-2k)_{2k} P^{2k-s,n-2k}_{j-k,\ell} (x)
       + q(\|x\|^2)  Y^{n-2j}_{\ell}(x),
\end{align} 
where $q\in \Pi^1_{j_0-k-1}$, in particular, 
$q=0$ if $s+k\ge j$.
\end{lem}

For $s =1$, the inner product \eqref{eq:ipd-s} becomes 
$$
  \la f, g\ra_{-1} = \la \nabla f, \nabla g\ra_{\ball} + \lambda_0 \la f,  g\ra_{\sph}. 
$$

\begin{thm} \label{thm:OPnabla}
A mutually orthogonal basis of $\CV_n^d(\varpi_{-1})$ is given by 
$\{P_{j,\ell}^{-1,n}(x): 0\le j \le \frac{n}{2}, 0\le \ell \le a_{n-2j}^d\}$ with    
\begin{align}\label{orth:P-1}
  h^{-1}_{j,n}: = \la P_{j,\ell}^{-1,n}, P_{j,\ell}^{-1,n}\ra_{-1} =   2d (n+\tfrac{d}{2}-1) (1-\delta_{j,0})  + (\lambda_0+dn)  \delta_{j,0}.
\end{align}
In particular, the space $\CV_n^d(\varpi_{-1})$ can be decomposed as 
$$
   \CV_n^d(\varpi_{-1}) = (1-\|x\|^2) \CV_{n-2}^d(\varpi_1)  \oplus \CH_n^d.
$$
\end{thm}

This theorem was first established in \cite{X08}, where the polynomials $P_{j,\ell}^{-1,n}(x)$ for $j \ge 1$ are 
written in the form
\begin{align}\label{eq:basisI}
 P_{j,\ell}^{-1,n}(x) = \frac{(2n+d-2)(2n+d-4)}{2j(2n-2j+d-2)} (\|x\|^2-1) P_{j-1,\ell}^{1,n-2}(x), 
      \quad 1 \le j \le \frac{n}{2},
\end{align}
which follows from \eqref{PN2P}. 
Recall that polynomials in $\CV_n^d(\varpi_{\mu})$ are eigenfunctions of a second order differential operator 
$\CD_\mu$ for $\mu > -1$. It turns out that polynomials in $\CV_n^d(\varpi_{-1})$ are eigenfunctions of 
$\CD_{-1}$, which explains our notation of $\varpi_{-1}$. For $s \ge 2$, however, $\CV_n^d(\varpi_{-2})$ is 
closely related, but not exactly, the space of eigenfunctions of $\CD_{-2}$; see the discussion in \cite{PX}. 

In the case of $s =2$, the inner product becomes 
$$
  \la f,g\ra_{-2}  = \la \Delta f, \Delta g\ra _{\ball} + \lambda_0 \la f, g \ra_{\sph}. 
$$

\begin{thm} \label{thm:OPdelta}
A mutually orthogonal basis of mutually orthogonal basis for $\CV_n^d(\varpi_{-2})$ is 
given by 
\begin{align*}
Q_{0,\ell}^n(x) & = Y_\ell^n(x), \quad Q_{1,\ell}^n(x) = (1-\|x\|^2) Y_\ell^{n-2}(x),\\
Q_{j,\ell}^n(x) & = (1-\|x\|^2)^2 P_{j-2}^{(n-2j+\frac{d-2}{2})}
     Y_\ell^{n-2j}(x), \quad 2 \le j \le \frac{n}{2},
\end{align*}
where where $\{Y_\nu^{n-2j}: 1 \le \nu \le a_{n-2j}^d\}$ is an orthonormal basis of $\CH_{n-2j}^d$. 
In particular, the space $\CV_n^d(\varpi_{-2})$ satisfies a decomposition 
\begin{equation} \label{eq:decom-s=-2}
     \CV_n^d(\varpi_{-2}) = (1-\|x\|^2)^2 \CV_{n-4}^d(\varpi_2)  \oplus (1-\|x\|^2) \CH_{n-2}^d \oplus \CH_n^d.
\end{equation}
\end{thm}
 
The decomposition in the theorem is established in \cite{PX}, from which the mutually orthogonal basis follows 
from results in \cite{X06}. The polynomials $Q_{j,\ell}^n$ are closely related to $P_{j,\ell}^{-2,n}$ as can be 
seen by \eqref{PN2P}. Indeed, $ P_{j,\ell}^{-2,n}(x) = c\, Q_{j,\ell}^n(x)$ except when $j =1$, in which case 
$$
     P_{1,\ell}^{-2,n}(x) = c_{1,\ell} Q_{1,\ell}^n(x) - \frac{n + \f d 2-1}{n + \f d 2-3} Y_\ell^{n-3}(x), \quad 2n+d > 6 
$$
and $ P_{1,\ell}^{-2,n}(x) = c_{1,\ell} Q_{1,\ell}^n(x) +2$ for $n = d =2$, where $c_{j,\ell}$ is a constant that 
can be obtained by comparing the leading coefficients. 

The orthogonal structure of $W_2^s(\ball)$ for $s \ge 3$ is more complicated. As we shall see below, the
analogue of \eqref{eq:decom-s=-2} no longer holds if $s > 2$. We start with a lemma. 

\begin{lem} \label{lem:Ynj}
Let $k$ and $n$ be nonnegative integers. For $Y \in \CH_n^d$,  
\begin{equation} \label{DeltaY}
  \Delta^k \big[  (1-\|x\|^2)^j Y(x) \big] \Big\vert_{x=\xi} = 
           4^k (-j)_k (-k)_{j-k}  \frac{(n+\frac{d}{2})_{k}}{(n+\frac d 2)_j} Y(\xi).
\end{equation}  
\end{lem}

\begin{proof} It suffices to prove \eqref{DeltaY} for all $Y(x)=Y^n_{\ell}(x)$, $\ell=0,1,\dots,a_{n}^d$.
By \eqref{GbaseP}, \eqref{JacobiP1}, \eqref{PN2P} and \eqref{LaplaceP}, we obtain 
\begin{align*}
  \Delta^k [ (1-\|x\|^2)^j Y^{n}_{\ell}(x)]  & =  \Delta^k [ (1-\|x\|^2)^j P^{j,n}_{0,\ell}(x)]  = \frac{(-j)_j }{(n+j+\tfrac{d}{2})_j} \Delta^{k}P^{-j,n+2j}_{j,\ell}(x)
  \\
  & = 4^k (-j)_j \frac{(n+2j+\tfrac{d}{2}-2k)_{2k}}{(n+j+\tfrac{d}{2})_j}P^{2k-j,n+2j-2k}_{j-k,\ell}(x),
\end{align*}
which is equal to zero if $0 \le j  \le k -1$ and its restriction on $\sph$ is zero if $j\ge 2k+1$ by 
\eqref{PN2P}.
It is easy to see that the right hand side of \eqref{DeltaY} is also zero
for $j$ in these ranges. In the remaining case $k \le j \le 2k$, we use \eqref{GbaseP} and \eqref{BndP} to derive 
\begin{align*}
  P^{2k-j,n+2j-2k}_{j-k,\ell}(\xi) 
=  \frac{(n+j-k+\frac{d}{2})_{j-k}   (2k-j+1)_{j-k} }{ (j-k)!  (n+k+\tfrac{d}{2})_{j-k} } Y^n_{\ell}(\xi),
\end{align*} 
and  simplify the constant by 
$$
 (-j)_j \frac{(n+2j+\tfrac{d}{2}-2k)_{2k}}{(n+j+\tfrac{d}{2})_j} \frac{(n+j-k+\frac{d}{2})_{j-k}   (2k-j+1)_{j-k} }{ (j-k)!  (n+k+\tfrac{d}{2})_{j-k} }
= (-j)_k (-k)_{j-k}  \frac{(n+\frac{d}{2})_{k}}{(n+\frac d 2)_j}.
$$
Then  \eqref{DeltaY} is established.
\end{proof} 

For $i=0,1,\ldots, j$, consider the system of linear equations
\begin{align}
\label{eq:system}
     4^k \sum_{i=k}^j (-i)_k  (-k)_{i-k}  \frac{(n+d/2)_k}{(n+d/2)_{i-k}} c_i =  \delta_{k,j}, \quad 0\le k\le j.
\end{align}
The system has a unique solution, since the matrix of the system is tridiagonal with nonzero diagonal elements. 
In fact, $c_0 = \delta_{j,0}$, $c_j = (-1)^j 4^{-j} / (j! (n+d/2)_j)$ and the rest $c_i$ can be deduced recursively starting from 
$c_j$.

\begin{lem} \label{def:Ynj}
For any $n,j\in \NN_0$, let   $c_{i}^{n,j}$, $0\le i \le j$, be the unique solution of the linear 
system \eqref{eq:system}.  If $j < 0$, define $Y^{n, j}_{\ell}(x):=0$ and, if $j \ge 0$, define 
\begin{align*}
   Y_{\ell}^{n,j}(x) := \sum_{i=0}^j c_i^{n,j}  (1-\|x\|^2)^i  Y^n_{\ell}(x),    \quad 1\le \ell \le a_n^d. 
\end{align*}
Then for $s \in \NN_0$, $x \in \ball$ and $\xi \in \sph$,  
\begin{align} \label{eq:Ynj}
  \Delta^s    Y_{\ell}^{n,j}(x)   = Y^{n,j-s}_{\ell}(x) \quad \hbox{and}\quad
  \Delta^s    Y_{\ell}^{n,j}(\xi) = \delta_{s,j}Y_\ell^n(\xi). 
\end{align}  
\end{lem}

\begin{proof}
The second identity of \eqref{eq:Ynj} follows directly from the definition of $c_{i,j}^n$ and the 
Lemma \ref{lem:Ynj}. To prove the first identity, we use the spherical-polar coordinates $x = \rho\xi$
and derive from \eqref{eq:Delta} and \eqref{eq:LaplaceBeltrami} that 
\begin{align} \label{eq:Delta-pola}
\Delta [q(\|x\|^2) Y^n_{\ell}(x) ]
=\, & q(\rho^2) \rho^{n-2} \Delta_0 Y^n_{\ell}(\xi)
+\big(\partial_\rho^2+ \tfrac{d-1}{\rho}\partial_\rho \big) q(\rho^2) \rho^n  Y^n_{\ell}(\xi) 
 \\
 =\, & Dq(\|x\|^2) Y^n_\ell (x),   \notag
\end{align}
where $Dq$ is defined by 
$$
 Dq(t) := 4 \big[t q''(t) + (n+\tfrac{d}{2}) q'(t)   \big].
$$
If $q$ is a polynomial of degree $j$, then $Dq$ is a polynomial of degree $j-1$. In particular, 
this shows that  $\Delta^{j+1} Y^{n,j}_{\ell}(x)=0$. Consequently, each $Y^{n,j}_{\ell}(x)$ is 
the solution of the following elliptic equation 
\begin{align}  \label{eq:aux0}
 \begin{cases}
    \Delta^{j+1} u =  0 & \text{ in } \ball,
    \\
    \Delta^{k} u = \delta_{j,k} Y^{n}_{\ell}&  \text{ on } \sph, \qquad k=0,1,\dots,j,
   \end{cases} 
\end{align}       
which admits a unique solution by the  posedness and regularity theory of the elliptic equation 
\cite[Thoerem 5.5.2, pp.\,390-391]{Triebel}. On the other hand, it is easy to see that $\Delta^s Y^{n,j+s}_{\ell}(x)$
is also a solution of \eqref{eq:aux0}. By the uniqueness of the solution, we must have
$\Delta^{s} Y^{n,j+s}_{\ell}(x)=Y^{n,j}_{\ell}(x)$, which completes the proof. 
\end{proof}

We did not find a closed-form formula for of $c_{i}^{n,j}$. Here are the first three $Y^{n,j}_\ell$: 
\begin{align*}
&Y^{n,0}_{\ell}(x) = Y^n_{\ell}(x),\qquad 
Y^{n,1}_{\ell}(x) = \frac{1-\|x\|^2}{n+\f d 2}Y^n_{\ell}(x),\\
&Y^{n,2}_{\ell}(x) = 
\frac{ (n+\f d 2) (1-\|x\|^2)^2 + 2 (1 - \|x\|^2)}{2 (n+\f d 2)(n+\f d 2)_2}Y^n_{\ell}(x).
\end{align*}

With the help of $Y_{\ell}^{n,j}$, we can now define a mutually orthogonal basis for $\CV_n^d(\varpi_{-s})$. 

\begin{lem} \label{lm:defQ}
For $s\in \NN$, $n\in \NN_0$, $0 \le j \le \frac{n}{2}$ and $1\le \ell \le a_{n-2j}^d$, define
\begin{align} \label{eq:defQ}
 Q^{-s,n}_{j,\ell}(x) =\begin{cases} P^{-s,n}_{j,\ell}(x), & j \ge s, \\
   P^{-s,n}_{j,\ell}(x) - \displaystyle{\sum_{k=0}^{\rhow-1}}
     \frac{\Delta^{k}P^{-s,n}_{j,\ell}(\xi) }{Y^{n-2j}_{\ell}(\xi)}Y^{n-2j,k}_{\ell}(x), & \rhow\le  j\le s-1,
\\
    Y^{n-2j,j}_{\ell}(x), &  0\le  j\le \rhow-1,
\end{cases}    
\end{align}
where $\xi \in \sph$. Then, for $0\le j \le \frac{n}{2}$ and  $1\le \ell \le a_{n-2j}^d$, 
\begin{enumerate}[    1. ]
\item $\Delta^{k} Q^{-s,n}_{j,\ell}(\xi) =\delta_{j,k} Y_{\ell}^{n-2j}(\xi)$ for $0\le k\le \rhow-1$; 
\item If $s$ is even, then $\Delta^{\rhoh} Q^{-s,n}_{j,\ell}(x) = 2^s (n+\frac{d}{2}-s)_s  P^{0,n-s}_{j-\frac{s}{2},\ell}(x)$;
\item If $s$ is odd, then $\Delta^{\rhoh} Q^{-s,n}_{j,\ell}(x) =2^{s-1} (n+\frac{d}{2}-s+1)_{s-1}  P^{-1,n-s+1}_{j-\frac{s-1}{2},\ell}(x)$ for $j \ne \rhoh$ and 
$\Delta^{\rhoh} Q^{-s,n}_{\rhoh,\ell}(x) =Y_\ell^{n-2\rhoh}(x)=P^{-1,n-s+1}_{0,\ell}(x)$. 
\end{enumerate}
\end{lem}

\begin{proof}
For $j\ge s$, it follows from \eqref{LaplaceP} that $\Delta^{k} P^{-s,n}_{j,\ell}(x) = 4^k (n+\frac{d}{2}-2k)_{2k}
P^{2k-s,n-2k}_{j-k,\ell}(x)$, which instantly  gives item 2 and item 3 for $s\le j \le \frac{n}{2}$. Further, 
 for $0\le k\le  \lceil \frac{s}{2} \rceil-1$, one derives from \eqref{PN2P}  that 
$\Delta^{k} P^{-s,n}_{j,\ell}(\xi)  = 0$ owing to $j-k \ge s-2k \ge 1$.
Hence, item 1 follows if $s\le j \le \frac{n}{2}$.

A  combination of \eqref{LaplaceP}  and \eqref{GbaseP}  implies that 
$\frac{\Delta^{k}P^{-s,n}_{j,\ell}(\xi) }{Y^{n-2j}_{\ell}(\xi)}$ is a constant independent of $\xi$. 
Hence, $Q_{j,\ell}^{-s,n}$ is defined in such a way that $\Delta^{k} Q^{-s,n}_{j,\ell}(\xi) =0$
for $\rhow\le  j\le s-1$, as can be easily verified using \eqref{eq:Ynj}. Furthermore, by \eqref{eq:Ynj} 
and  \eqref{LaplaceP}, we obtain
\begin{align*}
&\Delta^{\rhoh} Q^{-s,n}_{j,\ell}(x) 
= \Delta^{\rhoh} P^{-s,n}_{j,\ell}(x) - \delta_{\rhow-1,\rhoh } 
\frac{\Delta^{\rhoh}P^{-s,n}_{j,\ell}(\xi) }{Y^{n-2j}_{\ell}(\xi)} Y^{n-2j}_{\ell}(x)   
\\
&= 4^{\rhoh} (n+\tfrac{d}{2}-2 \lfloor \tfrac{s}{2} \rfloor)_{2\rhoh}
\Big(
  P^{2\rhoh-s,n-2\rhoh}_{j-\rhoh,\ell}(x)  - \delta_{\rhow-1,\rhoh } 
\frac{P^{2\rhoh-s,n-2\rhoh}_{j-\rhoh,\ell}(\xi) }{Y^{n-2j}_{\ell}(\xi)} Y^{n-2j}_{\ell}(x)  
\Big). 
\end{align*}
By \eqref{PN2P} and $j-\rhoh \ge s-2\rhoh \ge 1$ for odd $s$, $P^{2\rhoh-s,n-2\rhoh}_{j-\rhoh,\ell}(\xi) =0$;  
thus $\Delta^{\rhoh} Q^{-s,n}_{j,\ell}(x) 
=4^{\rhoh} (n+\tfrac{d}{2}-2 \lfloor \tfrac{s}{2} \rfloor)_{2\rhoh}P^{2\rhoh-s,n-2\rhoh}_{j-\rhoh,\ell}(x)$,  
which proves item 2 and item 3 for $\rhow\le  j\le s-1$.

Finally, if $0 \le j \le \rhow -1$, all three items follow  directly from \eqref{eq:Ynj}. 
\end{proof}

\begin{thm} \label{thm:Hjn-rho}
The polynomials in $\{Q^{-s,n}_{j,\ell}(x): 0\le j\le \tfrac{n}{2}; 1\le \ell\le a_{n-2j}^d \}$ form a mutually 
orthogonal basis of $\CV_{n}^d(\varpi_{-s})$. More precisely, 
\begin{align} \label{eq:Hjn-rho}
 \big \langle Q^{-s,n}_{j,\ell}, Q^{-s,n'}_{j', \ell'} \big \rangle_{-s} = h^{-s}_{j,n} \delta_{n,n'} \delta_{j,j'}\delta_{\ell,\ell'}
\end{align}
for $\la \cdot,\cdot\ra_{-s}$ defined in \eqref{eq:ipd-s}, where 
\begin{align*}
 h^{-s}_{j,n}:=
\begin{cases}
2^{2s-1} d (n+\tfrac{d}{2}-s)_{s} (n+\tfrac{d}{2}-s+1)_{s-1} ,
&    j \ge \rhow,\\
d (n -2j)+ \lambda_j,  & j= \frac{s-1}{2},\\
\lambda_j,   &     0\le j< \frac{s-1}{2}. 
\end{cases}
\end{align*}
\end{thm}
\begin{proof} 
From item 1 of Lemma \ref{lm:defQ} and the orthonormality of $\{ Y^n_{\ell}: 1\le \ell \le a_n^d \}$,
it follows immediately that
\begin{align*}
  \sum_{k=0}^{\rhow-1} \lambda_{k}  \big \langle \Delta^k Q^{-s,n}_{j,\ell},  \Delta^k Q^{-s,n'}_{j',\ell'} \big \rangle_{\sph}
   =  \lambda_j   \delta_{n,n'} \delta_{j,j'}\delta_{\ell,\ell'},
   \quad   0\le j, j' \le  \left \lceil \tfrac s 2 \right \rceil-1
\end{align*}
and the left hand side is equal to zero  if
 $j \ge \left \lceil \tfrac s 2 \right \rceil$ or $j' \ge \left \lceil \tfrac s 2 \right \rceil$. Thus, we remain to consider
$J:=\la \nabla^s Q_{j,\ell}^{-s,n}, \nabla^s Q_{j',\ell'}^{-s,n'}\ra_{\ball}$. 

 If $s$ is odd,  we temporarily denote 
  $c_{j,n}=2^{s-1} (n+\tfrac d 2 - s+1)_{s-1}$ if $j \neq \frac{s+1}{2}$ 
and $c_{\frac{s-1}{2},n}=1$.  From  item 3 of Lemma \ref{lm:defQ},
we obtain   that  
$$
J
  = c_{j,n} c_{j',n'}  \la \nabla P_{j- \f{s-1}{2}, \ell}^{-1,n - s+1}, \nabla P_{j'- \f{s-1}{2}, \ell'}^{-1,n' - s+1}  
       \big \rangle_{\ball},
$$
which is equal to zero  if $j\le \frac{s-3}{2}$ or $j'\le \frac{s-3}{2}$, 
whereas
  it can be seen from \eqref{orth:P-1} with $\l_0 =0$ that 
\begin{align*}J  
     = \big[d (n-s+1) \delta_{\frac{s-1}{2},j} + 2d  (n+\tfrac{d}{2}-s)(c_{j,n})^2 (1-\delta_{\frac{s-1}
     {2},j} ) \big]  \delta_{n,n'} \delta_{j,j'}\delta_{\ell,\ell'}
\end{align*}
for $j,j' \ge\frac{s-1}{2}$.  As a result, this completes the proof of \eqref{eq:Hjn-rho} for odd $s$.

If $s$ is even, we obtain from  item 2 of Lemma \ref{lm:defQ}  that  
$$
J  = 4^s(n+\tfrac{d}2-s)_s (n'+\tfrac{d}2-s)_s  \la P_{j- \f{s}{2}, \ell'}^{0,n - s},  P_{j'- \f{s}{2}, \ell'}^{0,n' - s}  
       \big \rangle_{\ball}.
$$
It is obvious that $J=0$ if $j\le \frac{s}{2}-1$ or $j'\le \frac{s}{2}-1$. For  $j,j'\ge \tfrac{s}{2}$, it follows
from \eqref{eq:Hjn-mu} that 
\begin{align*}
J   =  \frac{  \frac{d}{2} [2^s (n+\tfrac{d}2-s)_s]^2  }{ n+\frac{d}{2}-s  } 
   \delta_{n,n'} \delta_{j,j'}\delta_{\ell,\ell'}, 
\end{align*}
which proves \eqref{eq:Hjn-rho} for even $s$. The proof is completed.
\end{proof}
 
Just as the projection operator defined for $\mu > -1$, we define the orthogonal projection operator 
$\proj_n^{-s}: W^{s}_2(\ball) \mapsto \CV_n^d(\varpi_{-s})$, where $s$ is a positive integer, by
\begin{align} \label{eq:proj-s}
\proj_n^{-s} f(x) := \sum_{0\le j \le \tfrac{n}{2}} \frac{1}{h^{-s}_{j,n}} \sum_{\ell =1}^{a_{n-j}^d} 
   \big \langle f, Q^{-s,n}_{j,\ell} \big \rangle_{-s}   Q^{-s,n}_{j,\ell}(x),
\end{align}

\begin{lem}
Let $f \in W_p^s(\ball)$ and $s \in \NN$. For $k =0,1,\ldots \rhow -1$, 
\begin{equation} \label{DeltBndproj}
   \Delta^k \proj_n^{-s} f( \xi) = \proj_{n-2k}^\CH  \Delta^k f (\xi), \quad \xi \in \sph, 
\end{equation}
\end{lem}  

\begin{proof} Form Lemma \ref{lm:defQ}, Theorem \ref{thm:Hjn-rho} and the orthonormality of 
$\{ Y^{n}_{\ell}\}$, it follows that
\begin{align*}
   \Delta^k  \proj_n^{-s} f(\xi) = & \sum_{0\le j\le n/2 }\frac{1}{h^{-s}_{j,n}} \sum_{\ell=1}^{a^d_{n-2j}}  \la f, Q^{-s,n}_{j,\ell} \ra_{-s}   \delta_{j,k} Y^{n-2j}_{\ell}(x)   
   \\
   =& \sum_{\ell=1}^{a^d_{n-2k}}  \la \Delta^k f, Y^{n-2k}_{\ell} \ra_{\sph}     Y^{n-2k}_{\ell}  (\xi)
   = \proj_{n-2k}^{\CH} \Delta^kf(\xi),
\end{align*}
for $0\le k< \frac{s-1}{2}$. In the remaining case of $s$ is odd and $k = \frac{s-1}{2}$, it follows from
Lemma \ref{lm:defQ}, Theorem \ref{thm:Hjn-rho} and  \eqref{orth:P-1} that 
\begin{align*}
   \Delta^k & \proj_n^{-s} f(\xi) =  \sum_{0\le j\le \tfrac{n}{2} }\frac{1}{h^{-s}_{j,n}} \sum_{\ell=1}^{a^d_{n-2j}}  \la f, Q^{-s,n}_{j,\ell} \ra_{-s}     \delta_{j,k} Y^{n-2j}_{\ell}(x)  
   \\
   =&\frac{1 }{d(n-2k)+\lambda_k } \sum_{\ell=1}^{a^d_{n-2k}}
    \big[   \la\nabla \Delta^k f, \nabla Y^{n-2k}_{\ell} \ra_{\ball}  +    \lambda_k  \la \Delta^k f, Y^{n-2k}_{\ell} \ra_{\sph}  \big]   Y^{n-2k}_{\ell}  (\xi)
\\
   =&\frac{1 }{d(n-2k)+\lambda_k } \sum_{\ell=1}^{a^d_{n-2k}}
    \big[  d \la  \Delta^k f, \partial_{\sn} Y^{n-2k}_{\ell} \ra_{\sph}  +    \lambda_k  \la \Delta^k f, Y^{n-2k}_{\ell} \ra_{\sph}  \big]   Y^{n-2k}_{\ell}  (\xi)
\\    
   =&\sum_{\ell=1}^{a^d_{n-2k}}
    \la \Delta^k f, Y^{n-2k}_{\ell} \ra_{\sph}    Y^{n-2k}_{\ell}  (\xi)
   =  \proj_{n-2k}^{\CH} \Delta^kf(\xi),
\end{align*}
where the third equality sign is derived using Green's formula. 
\end{proof}

\begin{lem} \label{lem:proj-f=0g}
Assume that $f \in W_p^s(\ball)$ satisfies $f(x) = (1-\|x\|^2)^s g(x)$. Then for $s = 1,2 \ldots$, 
\begin{equation} \label{eq:proj-f=0g}
    \proj_n^{-s} f(x) = (1-\|x\|^2)^s \proj_{n-2s}^s g(x), \qquad x \in \ball. 
\end{equation}
\end{lem}

\begin{proof}
The fact of $(1-\|x\|^2)^s$ in $f$ implies immediately from \eqref{eq:Delta-pola} that $\Delta^k f(\xi) = 0$ if $k = 0,1,\ldots, \rhow -1$,
so that $\la f, Q^{-s,n}_{j,\ell}\ra_{-s} = \la \nabla^s f, \nabla^s Q^{-s,n}_{j,\ell}\ra_{\ball}$. 
By \eqref{eq:defQ} and \eqref{eq:Ynj}, $\Delta^sQ^{-s,n}_{j,\ell} =\Delta^s P^{-s,n}_{j,\ell}$ for $0\le j \le \frac{n}{2}$. 
Applying 
Green's identity repeatedly and using \eqref{LaplaceP}, we then obtain 
\begin{align*}
\la f, Q^{-s,n}_{j,\ell}\ra_{-s}  &=(-1)^s \la f, \Delta^s Q^{-s,n}_{j,\ell}\ra_{\ball} =  (-1)^s\la f, \Delta^s P^{-s,n}_{j,\ell}\ra_{\ball} \\
  =&(-4)^s (n + \tfrac d 2 - 2s)_{2s} \la f,  P^{s,n-2s}_{j-s,\ell}\ra_{\ball} 
  =    \frac{ (-4)^s  s! (n + \tfrac d 2 - 2s)_{2s}}{(\frac{d}{2}+1)_s  } \la g,  P^{s,n-2s}_{j-s,\ell}\ra_{s},
\end{align*}
which is zero if $j < s$; while for $j\ge s$, it follows from
\eqref{eq:defQ},  \eqref{eq:Hjn-rho},  \eqref{PN2P} and \eqref{eq:Hjn-mu} that 
\begin{align*}
\frac{\la f, Q^{-s,n}_{j,\ell}\ra_{-s}   Q^{-s,n}_{j,\ell} (x)  }{h^{-s}_{j,n}}
= \frac{ \la g, P^{s,n-2s}_{j-s,\ell} \ra_{s}  (1-\|x\|^2)^s P^{s,n-2s}_{j-s,\ell}(x) }{h^{s}_{j,n}},
\end{align*}
which finally proves \eqref{eq:proj-f=0g}.
\end{proof}

\section{Approximation by polynomials on the ball}
\setcounter{equation}{0}

This section contains our main results on approximation in the Sobolev space on the ball and their proofs. 
To facilitate readers who are mainly interested 
in the results, we state our main theorems in the first subsection and give their proofs in subsequent 
subsections. 

\subsection{Main results}\label{sec:main-result}

Let $\eta \in C^\infty [0, \infty)$ be an admissible cut-off function. With respect to the inner product 
$\la \cdot,\cdot\ra_{-s}$ of $W_p^s(
\ball)$, we define
\begin{equation} \label{eq:Sn-1}
      S_n^{-s} f(x):= \sum_{k=0}^n  \proj_{k}^{-s} f(x) \quad \hbox{and} \quad 
        S_{n, \eta}^{-s} f(x):= \sum_{k=0}^{\infty} \eta \left( \f{k}{n} \right)  \proj_{k}^{-s} f(x).   
\end{equation}
By definition, $S_n^{-s} f \in \Pi_n^d$ and $S_{n,\eta}^{-s} f \in \Pi_{2n-1}^d$, and it is obvious that 
\begin{equation} \label{eq:Sn-1*}
      \la f- S_n^{-s}f , v\ra_{-s}= 0  \quad \hbox{and} \quad   \la f- S_{n,\eta}^{-s}f , v\ra_{-s}=0, 
      \quad \forall v\in \Pi_n^d.
\end{equation}

Our first result is on approximation in the space of $W_p^1(\BB^d)$. 

\begin{thm}\label{1stApprox}
Let $r \in \NN$. If $f \in W_p^{r}(\ball)$ and  $1<p <\infty$, then 
\begin{align} \label{1stApprox0}
&\|f - S_{n, \eta}^{-1}f\|_{p,\ball} \le  c  n^{- 1}  \sum_{i=1}^d E_{n-1} ( \partial_i  f)_{p,\ball} 
  \le cn^{- r} \| f\|_{W_p^{r}(\ball)},
\\  \label{1stApprox1}
&\| \partial_i f -  \partial_i S_{n, \eta}^{-1}f\|_{p,\ball} \le   c   E_{n-1} ( \partial_i  f)_{p,\ball}
 \le cn^{-r+1} \|\partial_i f\|_{W_p^{r-1}(\ball)}, \quad 1\le i\le d,
\end{align}
where $S_{n, \eta}^{-1}f$  can be replaced by $S_{n}^{-1}f$ if $p=2$.
\end{thm}

For $W_p^{s}(\ball)$, $s=2,3,\dots$, the errors are not directly bounded by $E_n(f)_{p,\ball}$
but we still have the order of convergence. 

\begin{thm}\label{nthApprox}
Let $r, s\in \NN$ and $r \ge s$. If $f \in W_p^{r}(\ball)$  and $1<p <\infty$, then, for $n \ge s$, 
\begin{align} \label{nth1}
 \| f -  S^{-s}_{n,\eta} f \|_{W_p^k(\ball)} \le  c n^{-r+k}  \|f\|_{W_p^r(\ball)},
 \quad k = 0, 1,\ldots, s, 
\end{align}
where $S_{n, \eta}^{-s}f$  can be replaced by $S_{n}^{-s}f$ if $p=2$.
\end{thm}

We denote by $\accentset{\circ}{W}_p^{s}(\ball)$ the subspace of $W_p^s(\ball)$ defined by 
\begin{equation}\label{Snom1circ}
  \accentset{\circ}{W}_p^{s}(\ball) := \left \{f\in  W_p^{s}(\ball):  \partial_\sn ^k f \big \vert_{\sph} =0, 
    k=0,1,\ldots, s-1\right \},
\end{equation}
where $\partial_\sn$ denote the normal derivative of $\sph$. For $s =1,2,\ldots$ and $1 \le  p \le \infty$, we define 
a semi-norm of $W_p^s(\ball)$ by
\begin{align} \label{semi-norm-ball}
  |f|_{W_p^s(\ball)} : = \Big(\sum_{\a \in \NN^d_0,\, |\a| = s} \|\partial^\a f\|_{p,\ball}^p \Big)^{1/p}.
\end{align}
If $f \in \accentset{\circ}{W}_p^{s}(\ball)$, then it can be shown (see the end of Appendix B) that, for $1 <  p < \infty$, 
\begin{equation} \label{eq:norm-zeroboundary}
\|f\|_{W^s_p(\ball)} \sim |f|_{W^s_p(\ball)}  \sim  \|\nabla^r f \|_{p, \ball} \sim \sum_{i=1}^d \|\partial_i^s f\|_{p,\ball}. 
\end{equation}

\begin{thm} \label{nthApprox=Wcirc}
Let $r, s \in \NN$ and  $k \in \NN_0$. If $f\in\accentset{\circ}{W}_p^{s}(\ball)\cap {W}_p^{r}(\ball)$
with $r\ge s$  and $1 < p < \infty$,
then, for $n \ge s$, 
\begin{align} \label{nth2}
\|f-{S}_{n,\eta}^{-s} f \|_{W_p^k(\ball)} & \le  c n^{-r+k}   \sum_{|\a| = s} |\partial^\a f |_{W_p^{r-s}(\ball)}^\circ  
   \le  c n^{-r+k}     \|  f \|_{W_p^{r}(\ball)}.
\end{align}
where $S_{n, \eta}^{-s}f$  can be replaced by $S_{n}^{-s}f$ if $p=2$.
\end{thm}

It should be pointed out that if $f\in\accentset{\circ}{W}_p^{s}(\ball)$, then $f(x) = (1-\|x\|^2)^s g(x)$ for
some $g \in W_p^s(\ball)$ and, by Lemma \ref{lem:proj-f=0g},
$$
 S_{n,\eta}^{-s} f(x) = (1-\|x\|^2)^s S_{n-2s,\eta}^{s} g(x).
$$

These results will be proved in the following subsections, where the following observation will be useful.
For $f \in W_p^s(\ball)$,  it follows immediately from applying the H\"older inequality on $\ball$ and $\sph$ that 
$$
| \la f, g\ra_{-s} |\le   \Big( \| \nabla^s f\|_{p,\ball}  +  \sum_{k=0}^{\rhow-1} \sqrt{\lambda_k}  \|\Delta^k f\|_{p,\sph} \Big)
 \Big( \| \nabla^s g\|_{p,\ball}  +  \sum_{k=0}^{\rhow-1}\sqrt{\lambda_k}  \|\Delta^k g\|_{p,\sph} \Big). 
$$
Using  Lemma \ref{lem:EquivNorm} and the inequality \eqref{eq:imbedding}, we obtain the following lemma.

\begin{lem}
For  $f \in W_p^s(\ball)$ and $g \in W_q^s(\ball)$, with $\f 1 p + \f 1 q =1$ and $1 < p < \infty$, 
\begin{equation}\label{eq:Holder}
   |\la f, g \ra_{-s}| \le c \| f \|_{W^{s}_p(\ball)}  \| g \|_{{W^{s}_p(\ball)}}.  
\end{equation}
\end{lem}

\subsection{Approximation by polynomials in $W_p^1(\ball)$}
According to Theorem \ref{thm:OPnabla}, if $f\in W^1_p(\ball)$, then we can decompose $f$
into two parts
\begin{align*}
  f(x) = (1-\|x\|^2) g(x) + h(x), \quad \hbox{where}\quad  \Delta h = 0.  
\end{align*}
Then it is readily checked by Lemma \ref{lem:proj-f=0g} that
\begin{align*}
  \proj_n^{-1}f(x) = (1-\|x\|^2) \proj_{n-2}^1 g(x) + \proj_n^{\CH} h(x).  
\end{align*}

For $f \in W_p^1(\ball)$, let $S_n^{-1} f $ be defined as in \eqref{eq:Sn-1}. The following lemma is 
essential for the proof of Theorem \ref{1stApprox}. 

\begin{thm} \label{main-lemma}
For $1 \le i \le d$ and $n= 1,2,\ldots$, 
\begin{equation}\label{eq:main-lemma}
 \partial_i S_n^{-1} f = S_{n-1}^0 (\partial_i f) \quad\hbox{and} \quad  \partial_i S_{n,\eta}^{-1} f 
   = S_{n-1,\eta}^0 (\partial_i f).
\end{equation}
\end{thm}

\begin{proof}
Since $\partial_i S_n^{-1}f \in \Pi_{n-1}^d$, it suffices to prove 
\begin{align}
\label{DSn-1}
\left \langle \partial_i S_n^{-1}f-\partial_if, v\right \rangle_{\ball} = 0, \quad \forall v\in  \Pi_{n-1}^d,
\end{align}
for the first identity of \eqref{eq:main-lemma}. From the Fourier expansion of $f$,
\begin{align*}
f = \sum_{m\ge 0} \sum_{0\le j\le m/2} \sum_{\ell=1}^{a^d_{m-2j}}\wh{f}^{-1,m}_{j, \ell} P^{-1,m}_{j, \ell},  
\qquad  \wh{f}^{-1,m}_{j, \ell} = (h^{-1,m}_{j, \ell})^{-1} \la f, P^{-1,m}_{j, \ell}\ra_{-1},
\end{align*}
it follows that 
\begin{align*}
&\partial_if-\partial_i S_n^{-1}f= \sum_{m \ge n+1}\sum_{0\le j\le m/2} \sum_{\ell=1}^{a^d_{m-2j}}\wh{f}^{-1,m}_{j, \ell}\partial_i  P^{-1,m}_{j, \ell}.
\end{align*}
If $j \ge 1$, by integration by part and \eqref{PN2P}, we obtain that, for $m\ge n+1$,
\begin{align*}
\big \langle \partial_i P^{-1,m}_{j, \ell}, v\big \rangle_{\ball} 
= c \big  \langle (1-\|x\|^2) P^{1,m-2}_{j-1,\ell}, \partial_i v \big  \rangle_{\ball}  = 0, \quad v\in \Pi_{n-1}^d.
\end{align*}
If $j =0$, then $P^{-1,m}_{0, \ell}(x) = c Y^m_{\ell}(x)$. Since $\partial_iY^m_{\ell} $ is a homogenous polynomial
of degree $m-1$, and  $\Delta \partial_iY^m_{\ell} =0$, it follows that $\partial_iY^m_{\ell}
\in \CH^{d}_{m-1} \subseteq \CV_{m-1}^d(\varpi_0)$, so that 
\begin{align*}
\big  \langle \partial_i P^{-1,m}_{0,\ell}, v \big \rangle_{\ball}  = 0, \quad m\ge n+1, \quad v\in \Pi_{n-1}^d.
\end{align*}
As a result, \eqref{DSn-1} holds, which proves the first identity of \eqref{eq:main-lemma}.

Since $\partial_i\proj_n^{-1}f=\partial_i(S_n^{-1}f - S_{n-1}^{-1}f) 
= S_{n-1}^{0}\partial_i f - S_{n-2}^{0}\partial_i f =  \proj_{n-1}^{-1} \partial_i f$,
the second identity of \eqref{eq:main-lemma} follows readily. 
\end{proof}
 
\noindent
{\it Proof of Theorem \ref{1stApprox}}.
For $1 \le i \le d$, by \eqref{eq:main-lemma}, Theorem \ref{thm:near-best}
and \eqref{Jackson}, we obtain
$$
   \|\partial_i f - \partial_i S_{n,\eta}^{-1} f\|_{p,\ball} = \|\partial_i f  - S_{n-1,\eta}^0(\partial_i f)\|_{p,\ball}
     \le c E_{n-1}(\partial_i f)_{p,\ball} \le c n^{-r+1} \| \partial_i f \|_{W_p^{r-1}(\ball)}, 
$$
which proves \eqref{1stApprox1}. Furthermore, by \eqref{DeltBndproj} and  \eqref{JacksonS2}, 
$$
  \|f - S_{n,\eta}^{-1} f\|_{p, \sph} = \| f - S_{n,\eta}^{\CH} f\|_{p, \sph} \le c n^{-r+1} \|f\|_{W_p^{r-1}(\sph)}.
$$
Putting these together then applying Lemma \ref{EquivNorm1} with $s=1$ and $\CF(f)=\|f\|_{p,\sph}$, 
Lemma \ref{EquivNorm1}  with $s=r$,  and $\CF(f)=\|f\|_{W^{r-1}_p(\sph)}$,  we obtain
\begin{align} \label{eq:der-est-s=1}
 \| f- S_{n,\eta}^{-1} f \|_{W^{1}_p(\ball)} & \le c n^{-r+1} \left[  \|\nabla f\|_{W_p^{r-1}} +  \|f\|_{W_p^{r-1}(\sph)}\right] 
        \le c n^{-r+1} \|f\|_{W_p^{r}}. 
\end{align}

To prove \eqref{1stApprox0}, we use the Aubin-Nitsche duality argument.
We define 
\begin{align*}
  g = f - S_{n,\eta}^{-1} f \quad \hbox{and}\quad  g^* = \begin{cases} |g|^{p-1}\mathrm{sign}(g), & g \neq 0,\\
    0, & g = 0.
  \end{cases}
\end{align*}
Then $g \in L^p(\ball)$, $g^*\in L^q(\ball)$ and $\|g\|_{p,\ball} \|g^*\|_{q,\ball}=\la g^*, g\ra_{\ball} = \|g\|_{p,\ball}^p = \|g^*\|_{q,\ball}^q$ 
with $\frac1{p}+\frac1{q}=1$.
Consider the following auxiliary  elliptic boundary value problem  
\begin{align*}
   -\Delta u = g^*  \text{ \  \  in \  \ }  \ball, \quad \text{ and } \quad    \partial_\sn u + u   = 0  \text{ \  \  on \  \ } \sph. 
\end{align*}
It admits a unique solution $u\in {W}^{2}_{q}(\ball)$ such that $\|u\|_{W^2_q}\le c \|g^*\|_q$ 
(\cite[Theorem 5.5.2, pp.\,390-391]{Triebel}). Let $\la \cdot,\cdot \ra_{-1}$ denote the inner product \eqref{eq:ipd-s}
with $\l_0 =d$. The equivalent variational form reads 
\begin{equation} \label{eq:variation-1}
    \la u,v\ra_{-1} = \la \nabla u, \nabla v \ra_{\ball} + d\la u, v\ra_{\sph}  =  \la g^*, v\ra_{\ball} \qquad \forall v \in W^1_p(\ball).
\end{equation}
Since $S^{-1}_{n,\eta}$ reproduces polynomials of degree $n$, it follows that
$$
  \big \langle  S_{\lfloor \f n 2 \rfloor, \eta}^{-1} u, g  \big \rangle_{-1}
      = \big \langle  S_{\lfloor \f n 2 \rfloor, \eta}^{-1} u,   u - S_{n,\eta}^{-1} u \big \rangle_{-1} =0,
$$ 
Consequently, by \eqref{eq:variation-1} with $v = g$, the H\"older inequality \eqref{eq:Holder},  and 
\eqref{eq:der-est-s=1} with $r =2$, we obtain 
\begin{align*}
 \la g^*, g\ra_{\ball} & = \la u ,g\ra_{-1} = \big \langle u- S_{\lfloor \f n 2 \rfloor, \eta}^{-1} u, g \big \rangle_{-1} \\
  &  \le  \|u -   S_{\lfloor \f n 2 \rfloor, \eta}^{-1} u \|_{W^1_q(\ball)} \|g\|_{W^1_p(\ball)}  
      \le c n^{-1} \| u\|_{W^2_q(\ball)}   \|g\|_{W^1_p(\ball)} .
\end{align*}
We then apply $\|u\|_{W_q^2} \le c \|g^*\|_{q,\ball}$ and \eqref{eq:der-est-s=1} again to obtain 
\begin{align*}
  \|g\|_{p,\ball}  =  \frac{ \la g^*, g\ra_{\ball}} {\|g^*\|_{q,\ball}}  \le c n^{-1} \|g\|_{W^1_p(\ball)} 
           \le c n^{-r} \|f\|_{W_p^r(\ball)}. 
\end{align*}
This completes the proof of \eqref{1stApprox0}.
\qed

\subsection{Approximation by polynomials in $W_p^{s}(\ball)$} 
By definition, $S_n^{-s} f$ is a polynomial in $\Pi_n^d$ and $S_{n,\eta}^{-s} f$ is a polynomial in $\Pi_{2n-1}^d$.
We need an analogue of Theorem \ref{main-lemma}. 

\begin{thm} \label{main-lemma3}
For $n, s \in \NN_0^d$ and $n \ge s$, 
\begin{align}
\label{DeltaS}
  &\Delta^{\rhoh} S_n^{-s} f = S_{n-2\rhoh}^{2\rhoh-s} \Delta^{\rhoh}f ,
   \quad \Delta^{\rhoh} S_{n,\eta}^{-s} f = S_{n-2\rhoh,\eta}^{2\rhoh-s} \Delta^{\rhoh}f.
\end{align}
\end{thm}

\begin{proof}
By the definition \eqref{eq:proj-s}, we recall that
 \begin{align*}
\proj_n^{-s} f(x) =  \sum_{0\le j\le n/2} \sum_{l=1}^{a_{n-2j}^d} \widehat{f}^{-s,n}_{j,\ell} Q^{-s,n}_{j,\ell}(x),
\qquad  \widehat{f}^{-s,n}_{j,\ell} = \frac{1}{h^{-s}_{j,n}} \la f, Q^{-s,n}_{j,\ell}\ra_{-s}.
\end{align*}
 
We first assume $s=2m$. By item 2 of  Lemma \ref{lm:defQ}, it follows that 
$$
  \Delta^m Q^{-s,n}_{j,\ell}(x) =2^s (n+\tfrac{d}{2}-s)_{s} P^{0,n-s}_{j-m,\ell}(x),
$$ 
which is equal to zero if $ j\le m-1$.  
By item 1 of Lemma \ref{lm:defQ}, $\Delta^k Q^{-s,n}_{j,\ell} (\xi) =0$  holds for $j\ge m$ and $0\le k \le m-1$,  it  then follows from \eqref{eq:ipd-s}, \eqref{eq:Hjn-rho} and \eqref{eq:Hjn-mu}
that 
$$
    \widehat{f}^{-s,n}_{j,\ell}\Delta^m Q^{-s,n}_{j,\ell}(x)
    = \frac{\Delta^m Q^{-s,n}_{j,\ell}(x)}{h^{-s}_{j,n}}  \la \Delta^m f, \Delta^m Q^{-s,n}_{j,\ell}\ra_{\ball}
    = \frac{ \la \Delta^m f, P^{0,n-s}_{j-m,\ell}\ra_{\ball} }{h^{0}_{j-m,n-s}} P^{0,n-s}_{j-m,\ell}(x),
$$
for any $j\ge m$.
Since $\{ P^{0,n-s}_{j-m,\ell}:  m\le j\le \tfrac{n}{2}, 1\le \ell\le a_{n-2j}^d \}$ is a mutually orthogonal basis
of $\CV^d_{n-s}(\varpi_{0})$, we conclude that 
\begin{align*}
    \Delta^{m}\proj_n^{-s} f(x)  = \proj_{n-s}^{0} \Delta^{m}f(x),
\end{align*}
which leads to \eqref{DeltaS} for even $s$.

Now we assume $s=2m+1$. The same argument of using Lemma \ref{lm:defQ}, \eqref{eq:Hjn-rho}  and  \eqref{orth:P-1}
shows that $ \Delta^{m}Q^{-s,n}_{j,\ell}(x)=0$ for $j\le m-1$ and 
\begin{align*}
    \widehat{f}^{-s,n}_{j,\ell} \Delta^{m}Q^{-s,n}_{j,\ell}(x) = &
    \frac{\Delta^m Q^{-s,n}_{j,\ell}(x)}{h^{-s}_{j,n}} \big[ \la \nabla \Delta^m f, \nabla \Delta^m Q^{-s,n}_{j,\ell}\ra_{\ball}  + \lambda_m  \la \Delta^m f, \Delta^m Q^{-s,n}_{j,\ell}\ra_{\sph}  \big] 
    \\
    =&
   \left. \frac{ \la \Delta^m f, P^{-1,n-2m}_{j-m,\ell}\ra_{-1} }{h^{-1}_{j-m,n-2m}}\right|_{\lambda_0=\lambda_m} P^{-1,n-2m}_{j-m,\ell}(x),
        \quad j\ge m.
\end{align*}
 This implies
\begin{align*}
\Delta^{m}\proj_n^{-s} f(x)  = \proj_{n-2m}^{-1} \Delta^{m}f(x)
\end{align*}
by Theorem \ref{thm:OPnabla}. Thus, \eqref{DeltaS} also holds for odd $s$.
\end{proof}

\medskip\noindent
{\it Proof of Theorem \ref{nthApprox}.}
If $s = 2m$, it follows from Theorem \ref{main-lemma3}, Theorem \ref{thm:near-best} and 
Corollary \ref{cor:Jackson} that 
\begin{align*}
 \|\nabla^{s} f - \nabla^{s} S_{n,\eta}^{-s} f\|_{p,\ball}  & = 
     \| \Delta^m  f  - S_{n-2m,\eta}^0(\Delta^m f)\|_{p,\ball}  \le c E_{n-s} (\Delta^m f)_{p,\ball} \\
& \le c n^{-r+s} \|\Delta^m f\|_{W_p^{r-s}(\ball)}
     \le  c n^{-r+s}\sum_{|\a|=s} \|\partial^{\a} f\|_{W_p^{r-s}(\ball)},
\end{align*}
whereas, if $s=2m+1$, we obtain form Theorem \ref{main-lemma3}, Theorem \ref{1stApprox},
and Corollary \ref{cor:Jackson} that
\begin{align*}
    \|\nabla^{s} f - \nabla^{s} S_{n,\eta}^{-s} f\|_{p,\ball} & = 
            \| \nabla \Delta^{m-1}  f  - \nabla  S_{n-s+1,\eta}^{-1}(\Delta^{m-1} f)\|_{p,\ball}  
             \le c \sum_{i=1}^d E_{n-s} (\partial_i\Delta^{m-1} f)_{p,\ball}   \\
   & \le  c n^{-r+s} \sum_{i=1}^d \|\partial_i \Delta^m f\|_{W_p^{r-s}(\ball)}
   \le  c n^{-r+s} \sum_{|\a|=s} \|\partial^{\a} f\|_{W_p^{r-s}(\ball)}.
\end{align*}
Moreover, for $0 \le k \le \rhow -1$, we obtain, by \eqref{DeltBndproj} and \eqref{eq:approxHarm},
\begin{align*}
  \|\Delta^{k} (f -    S_{n,\eta}^{-s}f) \|_{W^{s-2k-1/p}_p(\sph)} 
  &  =   \|  \Delta^{k}  f   - S^{\CH}_{n-2 k,\eta} (\Delta^{k} f)\|_{W^{s-2k-1/p}_p(\sph)}  
   \\
  & \le c n^{- r + s} \| \Delta^k  f \|_{W_p^{r-2k-1/p}(\sph)} .
\end{align*}

Putting these together and applying Lemma \ref{lem:EquivNorm} and \eqref{eq:EquivNorm1}, we obtain
\begin{align} \label{high-der2}
  \|f - S_{n,\eta}^{-s} f \|_{W^s_p(\ball)} &  \le c n^{-r+s} \left[ \|\nabla^s f \|_{W_p^{r-s}(\ball)} 
     + \sum_{k=0}^{\rhoh -1} \| \Delta^k  f \|_{W_p^{r-2k-1/p}(\sph)}\right] \\
      & \le c n^{-r+s} \|f\|_{W_p^r(\ball)},    \notag 
\end{align}  
which establishes \eqref{nth1}  for the case $k = s$. 

Next we consider the estimate \eqref{nth1} for $k = 0$. This requires the following formula: for $s = 1, 2,3,\ldots$
\begin{align}\label{high-Green}
(-1)^s \int_{\ball} & \nabla^s f(x) \nabla^s g(x) dx
      =   \sum_{j=0}^{\lfloor \f{s}2 \rfloor -1} \int_{\sph} \Delta^{s-j-1} f(\xi) \partial_\sn \Delta^j g(\xi) d\xi \\
     & -  \sum_{j=0}^{\lceil \f{s}2 \rceil -1}   \int_{\sph}  \partial_\sn \Delta^{s-j-1} f(\xi) \Delta^j g(\xi) d\xi
             + \int_{\ball}  \Delta^s f(x) \, g(x)dx, \notag
\end{align}
where $\partial_\sn =  \frac{d}{d\sn}$ denote the derivative in the radius direction. For $s =1$, the first
term in the right hand side is taken to be zero and \eqref{high-Green} is the classical Green's identity. 
For $s = 2, 3,\ldots$, we apply Green's identity repeatedly. 

We need the following auxiliary partial differential equations with boundary values, 
\begin{align}
\label{auxrho}
\begin{cases}
(-\Delta)^{s}u = v,  & \text{in } \ball,
\\
\Delta^{s-1-k}u =0, & \text{on } \sph, k=0,1,\dots,\rhoh-1,
\\
 \partial_\sn \Delta^{s-1-k} u - (-1)^s \Delta^{k} u =0,
  & \text{on } \sph, k=0,1,\dots,\rhow-1.
\end{cases}
\end{align}
Let $\la \cdot, \cdot \ra_{-s}$ be defined as in \eqref{eq:ipd-s} with all $\l_k = d$. Using 
\eqref{high-Green} with $f =u$ shows that 
\begin{align}
\label{varvarrho}
 \la u, g\ra_{-s} = 
\la \nabla^{s}u, \nabla^{s}g \ra_{\ball} +d \sum_{k=0}^{\rhow -1}  \la  \Delta^{k}u, \Delta^{k}g\ra_{\sph}  
       = \la v, g\ra_{\ball}, 
\end{align}
for $g\in W_q^{s}(\ball)$, where $\frac{1}{p}+\frac{1}{q}=1$ for $1<p<\infty$. If $v =0$, then \eqref{varvarrho} 
with $u = g$ shows that $\|u\|_{-s}:= \la u, u\ra_{-s}= 0$, which implies that $u \equiv 0$. This shows that the homogenous 
problem of \eqref{auxrho} has a unique solution $u=0$. Hence, by Theorem 5.4.4/2 and Theorem 5.5.1
in \cite{Triebel}, $\Delta^{s}$ is an isomorphic mapping from the space ${\CU}_q^{2s}(\ball)$ onto $L^q(\ball)$, 
where  
\begin{align*}
{\CU}_q^{2s}(\ball) :=
&\Big\{ u:\in {W}_q^{2s}(\ball):  \Delta^{s-1-k}u =0 \text{ on } \sph, k=0,1,\dots, \left \lfloor \f s 2 \right \rfloor -1,
\\
& \text{ and }   \partial_\sn \Delta^{s-1-k} u - (-1)^s \Delta^{k} u =0,
  \text{on } \sph, k=0,1,\dots, \left \lceil \f s 2 \right \rceil -1 \Big\}.
\end{align*} 

As in the proof of Theorem \ref{1stApprox}, we use duality arguments and define  
\begin{align*}
  g = f - S_{n,\eta}^{-s} f \quad \hbox{and}\quad  g^* = \begin{cases} |g|^{p-1}\mathrm{sign}(g), & g \neq 0,\\
    0, & g = 0.
  \end{cases}
\end{align*}
By the isomorphism property of  $\Delta^{s}$,  \eqref{auxrho} with $v = g^*$ admits a unique solution that satisifes
$\|u \|_{W_q^{2s}(\ball)} \le  \|g^*\|_{q,\ball}$  \cite[Theorem 5.5.1]{Triebel}. 

Since $S_{n,\eta}^{-s}$ reproduces polynomials of degree $n$, it follows that 
$$
  \big  \langle g, S_{\lfloor \frac{n}{2}\rfloor ,\eta}^{-s} u \big  \rangle_{-s} =  \big  \langle  f - S_{n,\eta}^{-s} u, 
       S_{\lfloor \frac{n}{2}\rfloor,\eta}^{-s}u  \big  \rangle_{-s} =0. 
$$
As a result, we derive from \eqref{varvarrho} with $v = g^*$, \eqref{eq:Holder} and \eqref{high-der2} with 
$r = s$ that 
\begin{align*}
\la g^*, g\ra_{\ball}  = & \la u, g \ra_{-s} = \big  \langle u- S_{\lceil \frac{n}{2} \rceil,\eta}^{-s} u, g \big  \rangle_{-s} \\
\le & \| u-S_{\lceil \frac{n}{2} \rceil,\eta}^{-s} u \|_{W^{s}_q(\ball)}
     \| g \|_{W^{s}_q(\ball)} \le   c n^{-s} \| u\|_{W^{2s}_q(\ball)} \|g\|_{W^{s}_q(\ball)}.
\end{align*}
We then apply $\|u\|_{W^{2s}_q(\ball)} \le c \|g^*\|_q$ and obtain, using  \eqref{high-der2} again, that 
\begin{align*}
  \|g\|_{p,\ball}  =  \frac{ \la g^*, g\ra_{\ball}} {\|g^*\|_{q,\ball}}  \le c n^{-s} \|g\|_{W^{s}_q(\ball)}  \le c n^{-r} \|f\|_{W_p^{r}}.
\end{align*}
By the definition of $g$, this establishes \eqref{nth1} for $k =0$. Finally, the case $0 < k < s$ of
\eqref{nth1} follows from Lemma \ref{lem:interD}. 
\qed 
  
\subsection{Proof of Theorem \ref{nthApprox=Wcirc}}

If $f \in  \accentset{\circ}{W}(\ball)$ or if $f$ is a radial function, then $\la f, g\ra_{-s} = \la \nabla^s f, \nabla^s g\ra_{\ball}$.
As a result, by Corollary \ref{cor:Jackson}, \eqref{high-der2} can be replaced by 
\begin{align*}
    \|f-S_{n,\eta}^{-s} f \|_{W^s_p(\ball)}   \le c n^{-r+s} \sum_{|\a| =s} |\partial^\a f |_{W_p^{r-s}(\ball)}^\circ. 
\end{align*}
The proof of the first inequality of \eqref{nth2} now follows exactly the proof of Theorem \ref{nthApprox}. The
second inequality of \eqref{nth2} follows from
$$
  |\partial^\a f |_{W_p^{r-s}(\ball)}^\circ \le c \|\partial^\a f \|_{W_p^{r-s}(\ball)}
      \le c \|f\|_{W_p^r(\ball)},
$$
which completes the proof. \qed
 
\section{Applications and numerical examples}
\setcounter{equation}{0}

To illustrate our results, we consider numerical solutions of two elliptic equations of the second 
and fourth order, respectively, on the unit ball, and we choose the spectral-Galerkin method 
using orthogonal polynomials on the ball. We will carry out a convergence analyses of the approximation
scheme in the Hilbert space and present numerical examples that illustrate our theorems.

\subsection{Second order equation}
We consider the non-homogenous boundary problem of the Helmholtz equation on the unit ball,
\begin{align}
\label{Helmholtz}
-\Delta u +\lambda u= f \text{ in } \ball, \qquad   \partial_{\sn} u+\eta u = g \text{ on } \sph,
\end{align} 
where the  constant $\lambda \ge 0, \eta\ge 0$ and  $\lambda+\eta >0$. Let 
$$
  \CA_1(u,v):=  \big[ \la \nabla u, \nabla v\ra_{\ball}  +d\eta  \la u, v\ra_{\sph}\big]+ \lambda \la u,v\ra_{\ball} 
     =  \la u, v\ra_{-1}+ \lambda \la u,v\ra_{\ball}
$$ 
In the variational formulation, solving \eqref{Helmholtz} is equivalent 
to find  $u\in W^1_2(\ball)$ such that
\begin{align}
\label{variational1}
  \CA_1(u,v) = & \la f, v\ra_{\ball} + d\la g, v\ra_{\sph},  \  v\in {W}^1_2(\ball),  
\end{align}
which, by the Lax-Milgram lemma \cite{Lax}, admits a unique solution that satisfies
\begin{align}
\label{stability1}
\|\nabla u\|^2_{2,\ball} + \lambda\|u\|_{2,\ball}^2 + d\eta \|u\|_{2,\sph}^2 \le  c( \|f\|^2_{2,\ball} +d \|g\|_{2, \sph}^2).
\end{align}

The spectral-Galerkin approximation to \eqref{Helmholtz} amounts to find $u_n\in \Pi_n^d$ such that
\begin{align}
\label{GSM2nd}
  \CA_1(u_n,v)  = \la f, v\ra_{\ball} +d  \la g, v \ra_{\sph}, \qquad v\in \Pi_n^d,
\end{align}
which has a unique solution that satisfies \eqref{stability1} with $u_n$ in place of  $u$.

By Theorem \ref{thm:OPnabla}, the orthogonal expansions of $u_n \in \Pi_n^d$ can be written as
$$
      u_n=\sum_{k =0}^n \sum_{0\le j \le n/2} \sum_{\ell=1}^{a^d_{n-2j}} \wh{u}^{k}_{j,\ell} P^{-1,k}_{j,\ell}.
$$
Substituting this expression into \eqref{GSM2nd} and setting $v = P^{-1,k}_{j,\ell}$, we obtain a 
linear system of equations on $\{\wh{u}^{k}_{j,\ell}\}$. The matrix $[\CA_1(P^{-1,k}_{j,\ell},P^{-1,k'}_{j',\ell'})]$
contains two parts. The first part, called stiff matrix, has been computed in \eqref{orth:P-1}. To evaluate the 
second part, called the mass matrix, we use the definition of $P^{-1,k}_{j,\ell}$ given in \eqref{GbaseP}. By 
\eqref{JacobiP1} and \eqref{GbaseP}, 
it is not difficult to check that
\begin{align*}
   P^{-1,n}_{0,\ell}  =Y^n_{\ell} \quad\hbox{and}\quad P^{-1,n}_{j,\ell} = P^{0,n}_{j,\ell} - P^{0,n-2}_{j-1,\ell}, \quad j\ge1.
\end{align*}
Hence, using \eqref{eq:Hjn-mu}, we obtain that 
\begin{align*}
\la P^{-1,n}_{j,\ell}, P^{-1,n'}_{j',\ell'} \ra_{\ball}
= \begin{cases}
\frac{d}{2n+d} (2-\delta_{j,0}), &   n=n',     j=j', \ell=\ell', \\ 
-\frac{d}{2n'+d} , &   n=n'+2, j = j' + 1, \ell=\ell', \\ 
-\frac{d}{2n+d} , &   n'=n+2, j' = j + 1, \ell=\ell', \\ 
0, & \text{otherwise}.
\end{cases}
\end{align*}
Thus, the stiff matrix is diagonal and the mass matrix is tridiagonal when the coefficients are arranged
appropriately.  

The convergence of this approximation scheme is given in the following theorem. 

\begin{thm}
Let $u$ and $u_n$ be the solutions of \eqref{Helmholtz} and \eqref{GSM2nd}, respectively. If 
$u\in W_2^{s}(\ball)$ with $s\ge 1$, then
\begin{align*}
   \|u-u_n\|_{W_2^k(\ball)} \le cn^{-s+k} \|u\|_{W_2^s(\ball)}, \quad 0\le k \le 1\le s.
\end{align*} 
\end{thm}

\begin{proof} 
Recall that $S_n^{-1}$ denote the $n$th partial sum of orthogonal 
expansion with respect to $\la \cdot, \cdot \ra_{-1}$. By \eqref{eq:Sn-1*}, \eqref{variational1} and \eqref{GSM2nd},
\begin{align*}
\CA_1(u_n - S_n^{-1} u, v) = &   \la u_n - S_n^{-1} u, v\ra_{-1} +\lambda \la u_n - S_n^{-1} u, v\ra_{\ball}
\\
= & \la u_n - u, v\ra_{-1} +\lambda \la u_n -  u, v\ra_{\ball}  + \lambda \la u- S_n^{-1}u, v\ra_{\ball}  \\
= & \lambda \la u- S_n^{-1}u, v\ra_{\ball},  \quad  v\in \Pi_n^d.
\end{align*}
Taking $v=u_n - S_n^{-1} u \in \Pi_n^d$, we obtain 
\begin{align*}
 \CA_1 (u_n-S_n^{-1}u,u_n-S_n^{-1}u)
& \le \lambda \| u-S_n^{-1}u\|_{2,\ball}  \| u_n-S_n^{-1}u\|_{2,\ball}
 \\
 &\le \frac12\lambda  \| u-S_n^{-1}u\|_{2,\ball}^2 + \frac12 \CA_1(u_n-S_n^{-1}u,u_n-S_n^{-1}u),
\end{align*}
which implies that
\begin{align*}
  \| \nabla (u_n-S_n^{-1}u)\|_{2,\ball}  +\sqrt{ d\eta} \| u_n-S_n^{-1}u\|_{2,\sph} +  \sqrt{\lambda} \|u_n-S_n^{-1}u\|_{2,\ball}  
  \le \sqrt{3 \lambda}  \| u-S_n^{-1}u\|_{2,\ball}.
\end{align*}
Thus, by Lemma \ref{EquivNorm1} and \eqref{1stApprox1}, for $s = 1,2,\ldots$, 
\begin{align*}
&  \| u_n-u\|_{W^1_2(\ball)} \le c \|u-S_n^{-1}u\|_{W^1_2(\ball)}       \le c n^{-s+1} \|u\|_{W_p^s}. 
\end{align*}
Furthermore, by a standard  Aubin-Nitsche argument, we also have 
\begin{align*}
  \| u_n-u\|_{2,\ball} \le c n^{-1}  \| u_n-u\|_{W^1_2(\ball)},
\end{align*}
where we omit the details. Together, the last two displayed inequalities complete the proof. 
\end{proof}

\subsection{Fourth order equation}
We consider the following fourth order elliptic equation on the unit ball,
\begin{align}
\label{Biharmonic}
   &\Delta^2 u - \lambda_1 \Delta u + \lambda_0 u = f, \quad  \text{in } \ball,
   \qquad 
   u = \partial_\sn u = 0, \quad  \text {on } \partial\ball. 
\end{align}
where the constants $\lambda_1,\lambda_0 \ge 0$ and, for simplicity, we consider homogeneous 
boundary. In the variational formulation, solving \eqref{Biharmonic} is equivalent to find 
$u\in \accentset{\circ}{W}^2_2(\ball)$ such that 
\begin{align}
\label{variational2}
& \CA_2(u,v):=\la \Delta u, \Delta v\ra_{\ball} + \lambda_1 \la \nabla u, \nabla v \ra_{\ball}  + \lambda_0 \la u, v \ra_{\ball} = \la f, v\ra_{\ball},\quad v\in \accentset{\circ}{W}^2_2(\ball).
\end{align}
Let the approximation space be $\accentset{\circ}{\Pi}_n^d:={\Pi}_n^d\cap \accentset{\circ}{W}^2_2(\ball)$. The 
spectral Galerkin approximation scheme for \eqref{Biharmonic} amounts to find $u_n\in \accentset{\circ}{\Pi}_n^d$ 
such that
\begin{align}
\label{GSM4th}
&\CA_2(u_n,v)=\la f, v\ra_{\ball},\qquad v\in \accentset{\circ}{\Pi}_n^d,
\end{align}
which has a unique and stable solution by the Lax-Milgram lemma \cite{Lax}. 

The convergence of this approximation scheme is given in the following theorem.  

\begin{thm} \label{ConvBi}
Let $u$ and $u_n$ be the solutions of \eqref{Biharmonic} and \eqref{GSM4th}, respectively. If 
$u\in W_2^{s}(\ball)$ with $s \ge 2$, then
\begin{align} \label{estimate-PDE4}
   \|u-u_n\|_{W_2^{k}(\ball)} \le cn^{-s+k} \sum_{|\a| = 2} \|\partial^\a u\|_{W_2^{s-2}(\ball)}, \quad 0\le k\le 2\le s.
\end{align} 
\end{thm}

\begin{proof}
From \eqref{variational2} and \eqref{GSM4th}, it follows that $\CA_2(u-u_n, v)=0$. Since 
$S_n^{-2}u\in  \accentset{\circ}{\Pi}_n^d$,  by \eqref{eq:Sn-1*}, we obtain 
\begin{align*}
\CA_2(S_n^{-2}u-u_n, v) &  = \CA_2(S_n^{-2}u-u_n, v)    
\\
& = \la S_n^{-2}u-u, v\ra_{-2}
+ \lambda_1 \la \nabla (S_n^{-2}u-u), \nabla v \ra_{\ball}  + \lambda_0 \la S_n^{-2}u-u, v \ra_{\ball}  
\\
& =  \lambda_1 \la \nabla (S_n^{-2}u-u), \nabla v \ra_{\ball}
+  \lambda_0 \la S_n^{-2}u-u, v \ra_{\ball}.
\end{align*}
Taking $v=S_n^{-2}u-u_n \in  \accentset{\circ}{\Pi}_n^d$ in \eqref{variational2}, the above inequality shows
that 
\begin{align*}
  \|\Delta (u_n-S_n^{-2}u)\|^2_{2,\ball} + & \lambda_1 \|\nabla (u_n-S_n^{-2}u)\|^2_{2,\ball}
  +  \lambda_0 \| u_n-S_n^{-2}u\|^2_{2,\ball}
  \\
\le& \lambda_1 \|\nabla (u-S_n^{-2}u)\|^2_{2,\ball} + \lambda_0 \| u_n-S_n^{-2}u\|^2_{2,\ball},
\end{align*}
which leads to, by  \eqref{eq:EquivNorm2}, 
\begin{align*}
\|u_n-u\|_{W_2^2(\ball)} & \le \|u-S^{-2}_nu\|_{W_2^2(\ball)}+  \|u_n-S^{-2}_nu\|_{W_2^2(\ball)}
\\
  &  \le  \|u-S^{-2}_nu\|_{W_2^2(\ball)}+ c \|\Delta(u_n-S^{-2}_nu)\|_{2,\ball} \le 
    c \|u-S^{-2}_nu\|_{W_2^2(\ball)}.
\end{align*}
Consequently, the estimate for $k=2$ of \eqref{estimate-PDE4} follows from \eqref{nth2} in Theorem \ref{nthApprox=Wcirc}. 

A standard dual argument can then be used to derive the error estimate in the case of $k =0$ an $k=1$
of \eqref{estimate-PDE4}, we omit details. 
\end{proof}

\subsection{Numerical results}

We consider examples for the equations \eqref{Helmholtz} and \eqref{Biharmonic}. 

\begin{exam} \label{exam1}
For the Helmholtz equation  \eqref{Helmholtz}, we give numerical results in two and three dimensions with 
$\lambda=1$ in the following settings:
\begin{itemize}
\item[(a)] $d=2$,  $f(x)=x_1(11-x_1^2-x_2^2)$ and $g(\xi)=2\eta \xi_1$
such that  $u(x)=3x_1-(x_1^2+y_1^2)x_1$;
\vspace*{0.2em}
\item[(b)] $d=3$,  $f(x)=\dfrac{4-x_1^2-x_2^2}{4+x_1^2+x_2^2 -4 x_1}$ and 
$g(\xi)=\dfrac{4(1-\xi_3^2)(\xi_1-4) + 16 \xi_1  }{(5-\xi_3^2-4 \xi_1)^2} + \dfrac{\eta (3+\xi_3^2)}{5-\xi_3^2 -4 \xi_1}$
such that $u=f$.
\end{itemize}
\end{exam}

Let $\{ t_{i}^{(\beta)}, \omega^{(\beta)}_i \}_{i=0}^{n}$ be the zeros and the corresponding Christoffel numbers
of the Jacobi polynomials  $P^{(0,\beta)}_{n+1}(t)$.
Set $ \rho_i = \sqrt{(t_i^{(d/2-1)}+1)/2}$, $\theta_i = \arccos t^{(0)}_i$  for $ i=0,\dots,n$
and  $\phi_j = \frac{2 j \pi}{2n+1}$ for $j=0,1,\dots,2n$. 
We report  the discrete maximum 
error $e_M(u-u_n)$ and the discrete $L^2$--error $e_{L^2}(u-u_n)$, defined by
\begin{align*}
   &e_M(f) = \max_{0\le k_1, k_2/2 \le n} \big|f(x_{k})\big|, 
  \quad  e_{L^2}(f) = \sum_{0\le k_1,k_2/2\le n}  \big|f(x_{k})\big|^2 \frac{\pi\,  \omega^{(0)}_{k_1} }{2(2n+1)}, 
\end{align*}
with the measuring points  $x_{k} = \big(\rho_{k_1} \cos (\phi_{k_2}), \rho_{k_1} \sin (\phi_{k_2})\big)$ in two dimensions
and by 
\begin{align*}
   & e_M(f) = \max_{0\le k_1,k_2, k_3/2 \le n} \big|f(x_{k})\big|, 
  \quad  e_{L^2}(f) = \sum_{0\le k_1,k_2,k_3/2\le n}    \big|f(x_{k})\big|^2   \frac{\pi\, \omega^{(1/2)}_{k_1}       \omega^{(0)_{k_2}}} {2\sqrt{2}(2n+1)},
\end{align*}
with  $x_{k} = \big(\rho_{k_1} \cos (\phi_{k_3}), \rho_{k_1}  \sin(\theta_{k_2}) \cos (\phi_{k_3})  , \rho_{k_1}  \sin(\theta_{k_2}) \sin (\phi_{k_3}) \big)$ in three dimensions.

Theoretically, the spectral-Galerkin approximation \eqref{GSM2nd} with any $n\ge 5$ recover  the exact solution of Example \ref{exam1} (a).
Figure \ref{fig1}  shows the maximum and the $L^2$--errors between the exact solution and the approximation solution
of \eqref{GSM2nd}. It is easy to see from Figure \ref{fig1} (a) that all the errors plotted  are close to the machine precision, and an exponential oder of convergence is found in Figure  \ref{fig1} (b).
These conclusions match our theoretical results. 

\begin{figure}[!ht]
\centering
\hfill\includegraphics[width=0.45\textwidth]{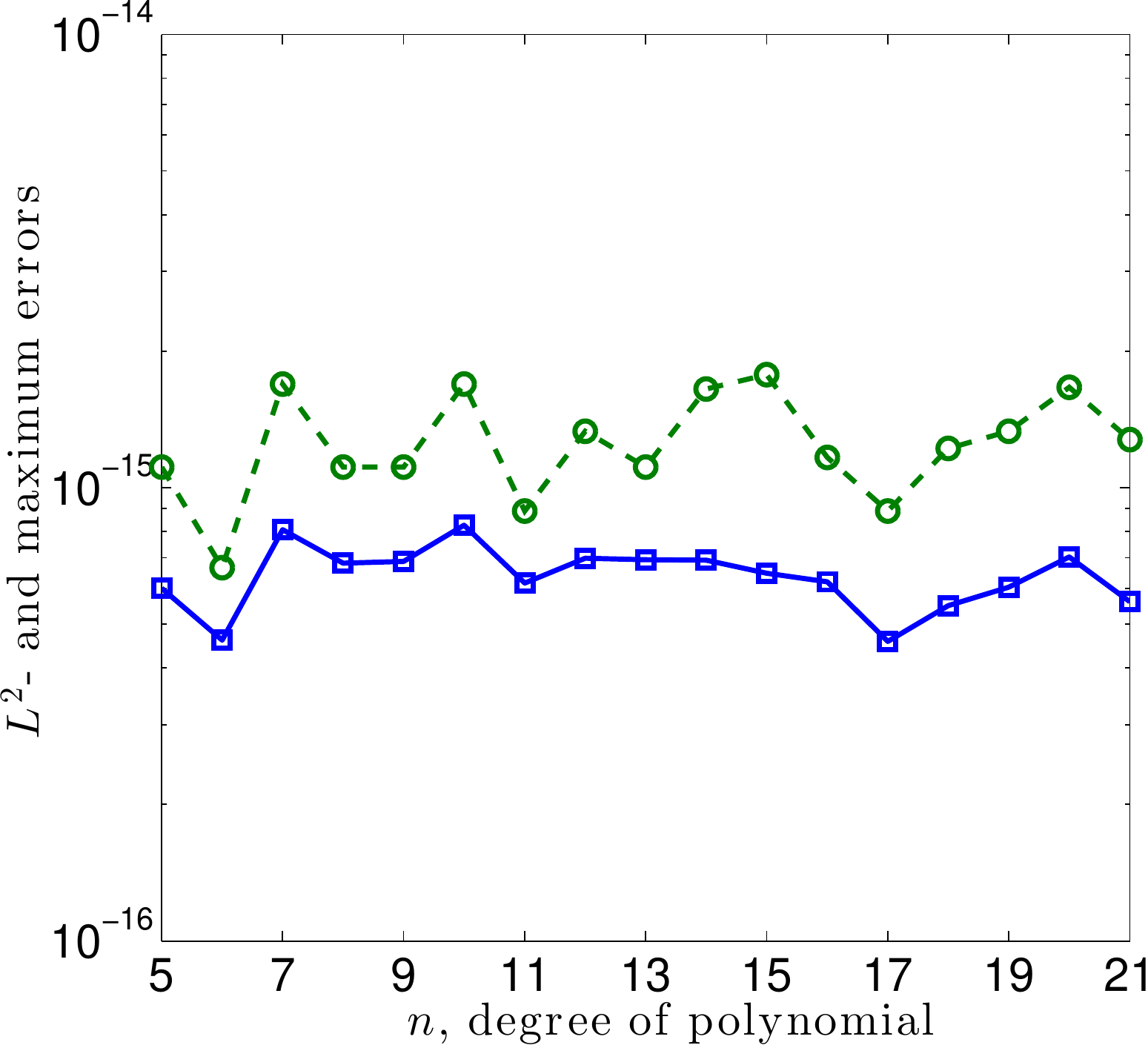}\hfill%
\hfill\includegraphics[width=0.45\textwidth]{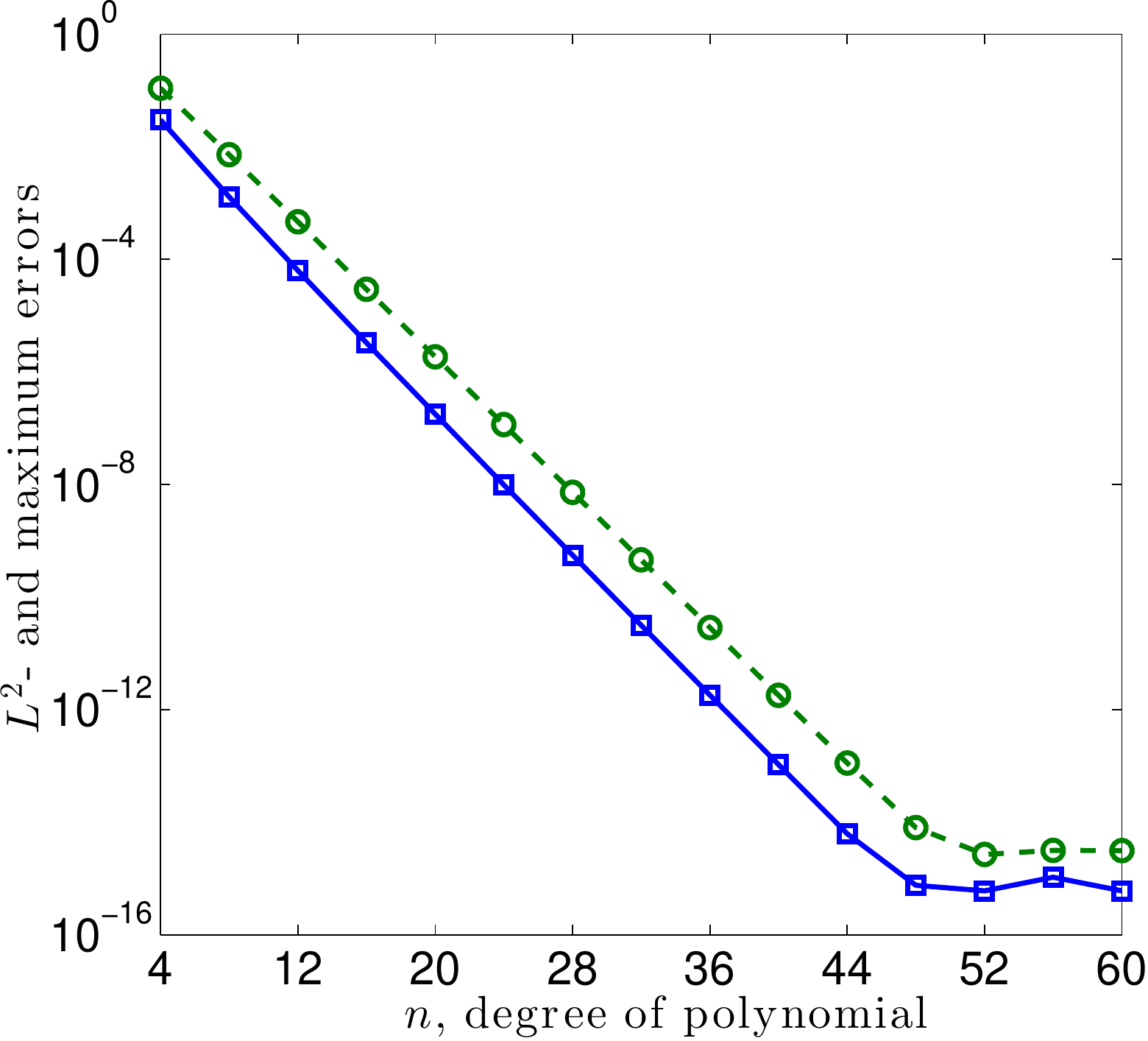}\hspace*{\fill}

\hfill(a)\hfill\hfill(b)\hspace*{\fill}

\caption{$L^2$-  (solid) and  maximum  (dashed) errors of Example \ref{exam1}.  Left:  $d=2$, $\eta=0$;  right: $d=3$, $\eta=1$. }\label{fig1}
\end{figure}

\

\begin{exam}
\label{exam2}
For the biharmonic equation \eqref{Biharmonic}, we consider an example with $\lambda_1=\lambda_0=1$, $d =2$
and the exact solution $u=\cos(2\pi(x^2+y^2))-1$. The function $f$ in the right hand side is determined by \eqref{Biharmonic}.
\end{exam}

Figure \ref{fig2} shows the maximum and the $L^2$- errors  of the approximation scheme \eqref{GSM4th}. 
An exponential order of convergence is observed in this plot, which is in agreement with Theorem 
\ref{ConvBi}.  
\begin{figure}[!ht]
\centering
\includegraphics[width=0.49\textwidth]{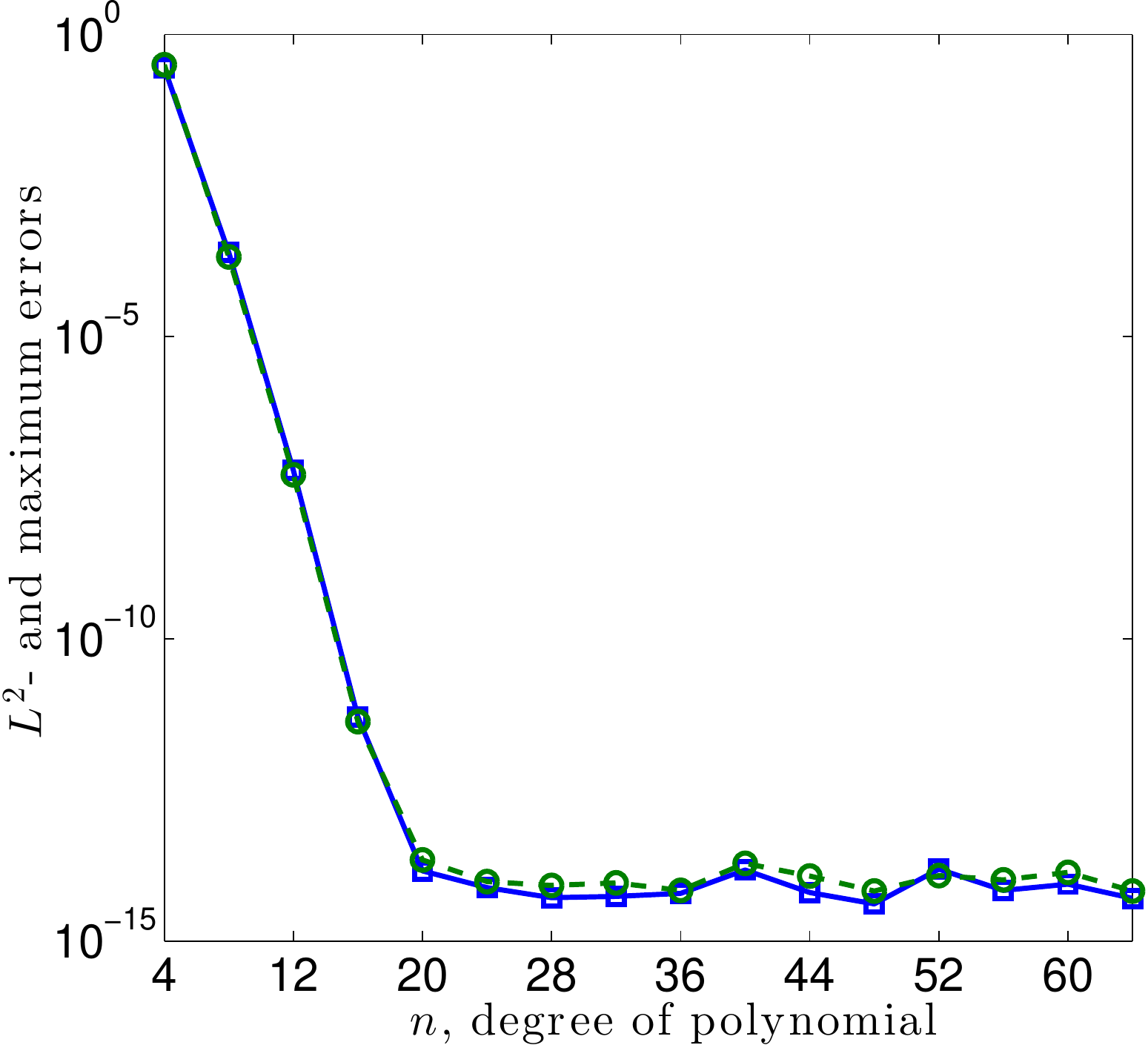}
\caption{$L^2$- (solid) and maximum  (dash) errors of Example \ref{exam2}.}
\label{fig2}
\end{figure}

\appendix
\section{Generalized orthogonal polynomials on the unit ball}\label{Gball}
\setcounter{equation}{0}

For $\a,\b>-1$, the Jacobi polynomials are defined by 
$$
  P_j^{(\a,\b)}(t) = \frac{(a+1)_j}{j!} {}_2F_1 \left(-j, j+\a + \b+a; \a+1; \frac{1-t}{2} \right).
$$
which are  orthogonal to each other with respect to the weight function 
$
   w^{\a,\b}(t) := (1-t)^\a (1+t)^\b
$
 on $[-1,1]$,
\begin{align}
\label{OrtJacP}
 \int_{-1}^1  P^{(\alpha,\beta)}_{j}(t) P^{(\alpha,\beta)}_{k}(t) w^{\a,\b}(t) dt
 = \frac{2^{\a+\b+1}  \delta_{j,k}}{2j+\a+\b+1} \frac{\Gamma(j+\a+1)\Gamma(j+\b+1)}{j!
 \Gamma(j+\a+\b+1)}.
\end{align}
As it is shown in \cite{Sz}, writing the ${}_2F_1$ in the following explicit form 
\begin{align}
\label{JacobiP}
   P^{(\alpha,\beta)}_{j}(t) 
   =& \sum_{k=0}^j  \frac{(k+\a+1)_{j-k} \, (j+\alpha+\beta+1)_{k} }{ (j-k)!\, k! } \left(\frac{t-1}{2}\right)^k,
\end{align}
extends the definition of $P_j^{(\a,\b)}(t)$ for all negative values of $\a$ and/or $\b$ in literature. 
However, if $-j-\a-\b\in\{ 1,2,\dots,j\}$, then a reduction of the degree of $P^{(\a,\b)}_j$ occurs. To 
avoid the degree reduction, we define the generalized Jacobi polynomials by
\begin{align}
\label{JacobiP1}
&\wh{P}^{(\alpha,\beta)}_{j}(t):=
\sum_{k=j_0}^j  \frac{(k+\a+1)_{j-k}  }{ (j-k)!\, k!  (j+\alpha+\beta+k+1)_{j-k} } \left(\frac{t-1}{2}\right)^k,   \quad j\in \NN_0,
\end{align}
where $j_0 = j^{\a,\b}_0(j):= -j-\a-\b$ if $-j-\a-\b\in\{ 1,2,\dots,j\}$ and $j_0 = 0$ otherwise. We also define
that ${P}^{(\alpha,\beta)}_{j}(t)=\wh{P}^{(\alpha,\beta)}_{j}(t)=0$ whenever $j$ is a negative integer. 
By the definition, it is evident that 
\begin{equation} \label{JacobiP2}
\wh{P}^{(\alpha,\beta)}_{j}(t):=
    \frac{1 }{(j+\a+\beta+1)_{j}}{P}^{(\alpha,\beta)}_{j}(t),
    \qquad \hbox{if $j_0= 0$}.
\end{equation}
The lemma below contains several properties of these polynomials, which are well--known properties of
the Jacobi polynomials if $j_0 =0$. 

\begin{lem} For $\a, \b \in \RR$, 
\begin{align} 
\label{BndP}
&\wh{P}_j^{(\a,\b)}(1) = \frac{(\a+1)_j}{j! (j+\a+\b+1)_j} \delta_{j_0,0},
\\
\label{JacPN}
&\wh{P}_j^{(\a,\b)}(t) = \frac{1}{(j+\a+1)_{-\a}} \Big(\frac{t-1}{2}\Big)^{-\a}  \wh{P}_{j+\a}^{(-\a,\b)}(t), \quad j\ge -\a\in \NN,
\\
\label{DiffP}
&\frac{d}{dt}\wh{P}_j^{(\a,\b)}(t) = \frac{1}{2}\wh{P}_{j-1}^{(\a+1,\b+1)}(t),  \qquad j\ge 0.
\end{align} 
\end{lem}

By comparison of the corresponding powers of $t-1$, both \eqref{JacPN} and \eqref{DiffP} follow from 
\eqref{JacobiP2}, and \eqref{BndP} is an immediate consequences of \eqref{JacobiP2}.
 
We now extend the definition of the orthogonal polynomials \eqref{baseP} on the unit ball to negative $\mu$,
\begin{defn}
Let $\mu \in \RR$. For $n \in \NN_0$ and $0 \le j \le \tfrac{n}{2}$, let $\{Y_\ell^{n-2j}: 1\le \ell\le a_{n-2j}^d\}$ be an 
orthonormal basis for $\mathcal{H}_{n-2j}^d$. Define 
\begin{align}
\label{GbaseP}
P^{\mu,n}_{j,\ell}(x) :=  (n-j+\tfrac{d}{2})_j  \wh{P}^{(\mu,n-2j+\tfrac{d}{2}-1)}_{j}(2\|x\|^2-1) Y^{n-2j}_{\ell}(x).
\end{align}
\end{defn}

We now prove Lemma \ref{DeltaP}, which we restate below.  

\medskip
{\noindent \bf Lemma \ref{DeltaP}}. \emph{Let $s\in \NN$ and $k,n\in \NN_0$. Then for $ 1\le\ell \le a_{n-2j}^d$,
\begin{align} 
&P^{-s,n}_{j,\ell}(x)
= \frac{(1-n-\tfrac{d}{2})_j}{(-j)_{s} (1-n-\tfrac{d}{2}+2s)_{j-s}}  (\|x\|^2-1)^{s} P^{s,n-2s}_{j-s,\ell}(x),  \quad   s \le j\le \tfrac{n}{2},
\tag{\ref{PN2P}}
\\
&       \Delta^{k} P^{-s,n}_{j,\ell} (x) = 4^{k} 
       (n+\tfrac{d}2-2 k)_{2 k} P^{2  k-s,n-2\k}_{j- k,\ell} (x)
       + p(\|x\|^2) Y^{n-2j}_{\ell}(x), \quad 0\le j \le \tfrac{n}{2}, \tag{\ref{LaplaceP}}
\end{align} 
where 
$p
\in \Pi^1_{j_0-k-1}$,
$j_0=j_0^{-s,n-2j+\frac{d}{2}-1}(j)$ and $P^{\mu,n}_{j,\ell}(x)=0$ if $j<0$ or $j>\frac{n}{2}$.
In particular,  $p=0$ if $j+k\ge s$. }

\begin{proof}
The identity \eqref{PN2P} is an immediate consequence of \eqref{JacPN} and \eqref{GbaseP}. Let $q$ be a polynomial
of degree $j$. Recall that, by \eqref{eq:Delta-pola},  
$$
\Delta [q(\|x\|^2) Y^n_{\ell}(x) ]=  Dq(\|x\|^2) Y^n_\ell (x)
$$
with $Dq(u) := 4 \big[u q''(u) + (n+\tfrac{d}{2}) q'(u)   \big]$. Let $\b=n-2j+\frac{d}{2}-1$, 
$q(u)=\wh{P}^{(-s,\beta)}_j(2u-1)$ and $t = 2u-1$. It follows from \eqref{JacobiP1} and \eqref{JacPN} that 
\begin{align*}
   Dq(u) & = 8  \big[ (\b+1)\partial_{t}  +(1+t) \partial_{t}^2 \big] \wh{P}_j^{(-s,\b)}(t)  \\
  & = 4\sum_{k=j_0}^j 
  \frac{  (k-s+1)_{j-k}   k  }{(j-k)!k!  (j-s+\b+k+1)_{j-k} }
  \Big[     (\b  + k)\frac{t-1}{2} + k-1 \Big]    \Big(\frac{t-1}{2}\Big)^{k-2}
\\
  &=  4(j+\b) \sum_{k=j_0}^j  \frac{  (k-s+2)_{j-k}  }{(j-k)!(k-1)! (j-s+\b+k+1)_{j-k} }   \Big(\frac{t-1}{2}\Big)^{k-1}
  \\
  &\quad +  \frac{4 (j_0-s+1)_{j-j_0}   }{(j-j_0)!(j_0-2)!  (j-s+\b+j_0+1)_{j-j_0} }
       \Big(\frac{t-1}{2}\Big)^{j_0-2}
\\
   & =  4(j+\b) \wh{P}_{j-1}^{(-s+2,\b)}(t) 
   + p_{j_0-2}(t),
\end{align*}
where $p_{j_0-2}\in \Pi^1_{j_0-2}$, and  the last equal sign is derived using the fact that
 $j_0 = j^{(-s,\b)}_0(j) =  j^{(-s+2,\b+1)}_0(j-1)+1$ if $j_0\ge 1$.
Putting above computations together, it follows that 
\begin{align*}
\Delta P^{-s,n}_{j,\ell}(x) 
& = (n-j+\tfrac{d}{2})_j \Delta \big[ \wh{P}^{(-s,n-2j+\tfrac{d}{2}-1)}_j(2\|x\|^2-1)Y^{n-2j}_{\ell}(x) \big]
\\
& =  4(n-j+\tfrac{d}{2}-1)_{j+1} 
\wh{P}^{(-s+2,n-2j+\tfrac{d}{2}-1)}_{j-1}(2\|x\|^2-1)Y^{n-2j}_{\ell}(x)  + Q_{j_0-2}(x)\\
& = 
 4(n+\tfrac{d}{2}-1)(n+\tfrac{d}{2}-2)P^{-s+2,n-2}_{j-1,\ell}(x)+ Q_{j_0-2}(x),
\end{align*}
where $Q_{j_0-2}(x) =(n-j+\tfrac{d}{2})_j   p_{j_0-2}(2\|x\|^2-1) Y^{n-2j}_{\ell}(x)$. Using this identity recursively, we 
derive 
\begin{align*}
\Delta^k P^{-s,n}_{j,\ell}(x) 
 4^k(n+\tfrac{d}{2}-2k)_{2k}P^{-s+2k,n-2k}_{j-k,\ell}(x)+ Q_{j_0-k-1}(x),
\end{align*}
where $Q_{j_0-k-1}(x) =p(2\|x\|^2-1) Y^{n-2j}_{\ell}(x)$ for certain $p\in \Pi_{j_0-k-1}^1$.
Specifically,   $  s-(n-2j+\frac{d}{2}-1)-j = s+j-n-\frac{d}{2}+1 \le k $
if $j+k\ge s$,  in return, $j_0\le k$ and $ \Pi^1_{j_0-k-1} \ni p = 0$. This completes the proof.
\end{proof}

Monic orthogonal polynomials in $\CV_n^d(\varpi_\mu)$ are defined by (cf. \cite[p. 42]{DX})
\begin{align} \label{monic}
 V_\alpha^\mu (x)=   \sum_\gamma 
  \frac{(-\a)_{2\g}}{(1-\mu-\tfrac{d}{2}-|\a|)_{|\g|} \g!}
 2^{-2 |\gamma|} x^{\a- 2\gamma}.
\end{align}
Since $(-\a_i)_{2\g_i} =0$ if $2\g_i > \a_i$, $V_\a^\mu$ is a polynomial of degree $|\alpha|$; 
in fact, $V_\a^\mu(x) - x^\a \in \Pi_{|\alpha|-1}^d$. Moreover, $\{ V_\alpha^\mu: |\a|=n, \, \a \in \NN_0^d\}$
is a basis of $\CV_n^d(\varpi_\mu)$. 
Moreover, $V_{\a}^{\mu}$ is well defined  
if and only if  $(1-\mu-\tfrac{d}{2}-|\a|)_{|\g|}\neq 0$ for all $\g\le \lfloor \a/2\rfloor$.
If there exists $\g_0\in \NN^d_0$  such that $(1-\mu-\tfrac{d}{2}-|\a|)_{|\g_0|}= 0$  and $2\g_0\le  \a$, 
we use  a truncated series for $V_\alpha^{-s} (x)$,
\begin{align}
\label{Gmonic}
 V_\alpha^{\mu} (x) =
 \sum_{|\g| \le \mu+\tfrac{d}{2}+|\a|-1  }   \frac{  (-\a)_{2\g} }
   { (1-\mu-|\a|-\tfrac{d}{2})_{|\g|}  \g!} 2^{-2 |\g|} x^{\a- 2\g},
\end{align}
which removes the lower order terms in $V_\a^{\mu} (x)$. 

\begin{lem} 
\label{lm:DiffV}For $\b\in \NN_0^d$ and $\mu \in \RR$ with $\mu + \frac d 2 +|\a|  \notin \{1,2,\ldots, \lfloor \f \a 2 \rfloor\}$, 
\begin{align}
\label{DiffV}
\partial^{\beta} V_\a^\mu (x)  =  (-1)^{|\beta|} (-\alpha)_{\beta}
 V_{\a-\beta}^{\mu+|\beta|} (x), \quad \alpha\in \NN_0^d.
\end{align}
\end{lem}
\begin{proof}
Taking derivative $\partial_i$ on \eqref{monic} or \eqref{Gmonic}, 
and using the fact that $(-\a_i)_{2\g_i} (\a_i-2\g_i) = \a_i   (1-\a_i)_{2\g_i}$, it is easy to see that 
\begin{align*}
  \partial_i V_\a^\mu (x) =&   
   \sum_\g     \frac{(-\a)_{2\g}}{(1-\mu-\tfrac{d}{2}-|\a|)_{|\g|} \g!} 2^{-2 |\gamma|}   (\a_i-2\g_i)x^{\a- 2\g-e_i}
         =  \a_i V_{\a-e_i}^{\mu+1}(x),
\end{align*}
which, when used recursively, leads to \eqref{DiffV}. 
\end{proof}
 

\section{Sobolev spaces}
\setcounter{equation}{0}

In this appendix we discuss equivalent norms of the Sobolev space $W_p^s(\ball)$. Several 
results that we shall need hold for fairly general domain $\Omega$ in $\RR^d$. We state
only their simplified version for $\Omega = \ball$ and/or $\Omega =\sph$. In this appendix, we
adopt the convention that $A \sim B$ means $c_1 A \le B \le c_2 A$ for some constants $c_2 > c_1 >0$.

For $s =1,2,\ldots$ and $1 \le  p \le \infty$, we define a semi-norm of $W_p^s(\ball)$ by
\begin{align*}
  &|f|_{W_p^s(\ball)} : = \Big(\sum_{\a \in \NN^d_0,\, |\a| = s} \|\partial^\a f\|_{L^p(\ball)}^p \Big)^{1/p}.
\end{align*}

\begin{lem}[\text{\cite[Theorem 5.12, p.143]{Adam}}] \label{lem:interD}
Let $0<\varepsilon_0< \infty$,  let $1\le p < \infty$,  and let $j$ and $s$ be integers with
$0 < j < s - 1$. There exists a constant $K = K(\varepsilon_0, s, p, d)$ 
 such that for every  $f \in W^s_p (\ball)$,
\begin{align*}
  |f|_{W^j_p(\ball)} \le K \varepsilon |f|_{ W^s_p(\ball)} + K 
         \varepsilon^{-j/(s-j)} \|f\|_{L^p(\ball)}, \quad 0 < \varepsilon \le \varepsilon_0.
\end{align*}
\end{lem}
 As a consequence of this lemma, it follows that 
\begin{align} \label{norm-equiv}
  & \| f\|_{W_{p}^{s}(\ball) } \sim  \|f\|_{L^p(\ball)}  +   |f|_{W_{p}^{s}(\ball)}
       \sim   \|f\|_{L^p(\ball)}   +\sum_{i=1}^d \|\partial_i^s f\|_{L^p(\ball)},
\end{align}
where the last equivalence signs  are derived from \cite[Thoerem 4.2.4, p.316]{Triebel}.

We need another lemma on equivalent norms in $W^s_p(\ball)$.

\begin{lem}\cite[Theorem 1.1.16]{mazia2011} \label{EquivNorm1}
Let $s = 1,2,\ldots$ and let $\CF(f)$ be a continuous seminorm in $W^{s}_p(\ball)$
such that $\CF(P_{s-1})\neq 0$ for any nonzero polynomial $P_{s-1}\in \Pi_{s-1}^d$.
 Then 
\begin{align*}
   \sum_{|\a|=s} \|\partial^{\a} f\|_{L^p(\ball)}  + \CF(f) \sim \|f\|_{W_p^{s}(\ball)}.
\end{align*}
\end{lem}
A combination of \eqref{norm-equiv}  with Lemma \ref{EquivNorm1} leads to 
\begin{align}
\label{eq:EquivNorm1}
\begin{split}
 \|f\|_{W_p^{m+s}(\ball)} 
 \sim& \sum_{|\b|=m+s} \|\partial^{\b} f\|_{L^p(\ball)}  
   + \bigg[ \sum_{|\a|=s} \|\partial^{\a} f\|_{L^p(\ball)}   + \CF(f)\bigg]
   \\
    \sim  &  \sum_{|\a|=s} \|\partial^{\a} f\|_{W^m_p(\ball)}  + \CF(f).
\end{split}    
\end{align}

We need the fractional order Sobolev space on $\sph$, which is defined via the interpolation space.
For $f\in W_p^{s_0}(\sph)+W_p^{s_1}(\sph)$ with $s_0,s_1\in \NN_0$, we define the $K$-functional 
$$
   K_{s_0,s_1}(f,t)_{p, \sph} : = \inf_{f=f_0+f_1} \left\{ \|f_0\|_{W_p^{s_0}(\sph)} + t \|f_1\|_{W_p^{s_1}(\sph)} \right\}.
$$
For $0 < \t <1$ and $0<\theta<1$, the fractional order Sobolev space $W^{s+\theta}_p(\sph)$ is defined 
as the interpolation space $(W^{s}_p(\sph), W^{s+1}_p(\sph))_{\theta,p}$ via the $K$-functional  \cite[\S1.3]{Triebel}, 
\begin{align*}
   W^{s+\theta}_p(\sph) := 
   \left\{ f\in W^{s}_p(\sph)+W^{s+1}_p(\sph): \|f\|_{W^{s+\theta}_p(\sph)}<\infty \right\},
\end{align*}
where the norm is defined by 
\begin{align*}
 \|f\|_{W^{s+\theta}_p(\sph)}:
   = \left( \int_{0}^{\infty} \big[ t^{-\theta} K_{s,s+1}(f,t)_{p, \sph}  \big]^p  \frac{dt}{t}\right)^{1/p}.
\end{align*}
It follows that $W^{s+1}_p(\sph)\subset W^{s+\theta}_p(\sph)\subset W^{s}_p(\sph)$ and, furthermore 
\cite[\S1.3]{Triebel}, 
\begin{align}
\label{eq:imbedding}
    &\|f\|_{W^{s}_p(\sph)} \le c \|f\|_{W^{s+\theta}_p(\sph)} , \qquad f  \in W^{s+\theta}_p(\sph).
\end{align}

It is worth to point out that the interpolator $(\cdot,\cdot)_{\theta,p}$ is of type $\theta$. Let 
$X_{\t}:=W^{r+\t}_p(\sph)$, $Y_{\t}:=W^{s+\t}_p(\sph)$ for $0\le \t\le 1$ and assume 
$T\in \CL(X_i, Y_i)$, $i=0,1$. It follows then that $T\in \CL(X_\theta, Y_\theta)$ and 
\begin{align}
\label{eq:theta}
  \|T\|_{\CL(X_\theta,Y_\theta)} 
  \le  \|T\|_{\CL(X_0,Y_0)} ^{1-\theta} 
   \|T\|_{\CL(X_1,Y_1)} ^{\theta}.  
\end{align}

Recall that $\nabla^{2m} = \Delta^m$ and $\nabla^{2m+1} = \Delta^m \nabla$. We define 
$\|\nabla f\|_{p,\ball} := |f|_{W_p^1}$ and define $\|\nabla^{2m+1} f\|_{p,\ball}$ using 
$\Delta^{2m+1} = \nabla \Delta^m$ accordingly. The following lemma is needed in the proof of main theorems. 

\begin{lem}\label{lem:EquivNorm}
For $f \in W_p^r(\ball)$ and $r \ge 1$, 
 \begin{align}\label{eq:EquivNorm2}
 \|\nabla^r f \|_{L^p(\ball)} 
      + \sum_{k=0}^{\lceil \tfrac{r}{2} \rceil-1} \| \Delta^k  f \|_{W_p^{r-2k-1/p}(\sph)} \sim \|f\|_{W_p^r(\ball)}.    
\end{align}
\end{lem}

\begin{proof}
  By Thoerem 5.5.2 in  \cite[p.\,391]{Triebel},   $\{ \Delta, I|_{\sph}\}$ is an isomorphic  mapping from $W^{2+s}_{p}(\ball)$
  onto $W^{s}_{p}(\ball)\times W^{2+s-1/p}_{p}(\sph)$, which means that 
\begin{align*}
  \|f\|_{W_p^{s+2}(\ball)} \sim    \|\Delta f \|_{W_p^{s}(\ball)} + \|f \|_{W_p^{2+s-1/p}(\sph)},
  \qquad\forall\, f\in W_p^{s+2}(\ball).
\end{align*}
It then follows from recursive reduction that, for any  $f\in W_p^{r}(\ball)$,
\begin{align*}
\| f\|_{W_p^{r}(\ball)} \sim &  
 \|\Delta^m f \|_{W_p^{r-2m}(\ball)}  + \sum_{k=0}^{m-1} \|\Delta^{k} f \|_{W_p^{r-2k-1/p}(\sph)},
\end{align*} 
which proves \eqref{eq:EquivNorm2} for $r=2m$. Furthermore, assuming $r=2m+1$ and taking 
$\CF( f) =  \| f  \|_{W_p^{1-1/p}(\sph)} $ and $s=1$ in Lemma \ref{EquivNorm1}, we can then deduce that
\begin{align*}
 \| f\|_{W_p^{r}(\ball)} \sim &  \| \nabla  \Delta^{m} f \|_{L^p(\ball)}   + \sum_{k=0}^{m} \|\Delta^k f \|_{W_p^{r-2k-1/p}(\sph)},
\end{align*}
which  proves \eqref{eq:EquivNorm2} for $r=2m+1$. The proof is completed. 
\end{proof}

In particular, for $f \in \accentset{\circ}{W}_p^{s}(\ball)$, \eqref{eq:EquivNorm2} together with
the inequality $\|\partial_i \partial_j f\|_{L^p(\ball)}  \le c \|\Delta f\|_{L^p(\ball)}$ for $1 < p <\infty$ imply 
the inequality \eqref{eq:norm-zeroboundary}.

\subsection*{Acknowledgement}
The authors would like to appreciate Professor K. Atkinson for his help in improving the presentation.

\end{document}